\documentclass[eqnor,reqno]{amsart}
\usepackage{amssymb,latexsym,amsmath,amsthm,enumerate,mathabx}
\usepackage[mathscr]{eucal}
\usepackage{framed,color,graphicx}
\usepackage{mathrsfs,fontenc}
\usepackage[all]{xy} 
\usepackage{tikz}
\usetikzlibrary{positioning}

\DeclareSymbolFont{bbold}{U}{bbold}{m}{n}
\DeclareSymbolFontAlphabet{\mathbbold}{bbold}

\theoremstyle{plain} 
\newtheorem{thm}{Theorem}[section] 
\newtheorem{lem}[thm]{Lemma}
\newtheorem{WilkiesTheorem}[thm]{Wilkie's HSI Theorem}
\newtheorem{cor}[thm]{Corollary}
\newtheorem{pro}[thm]{Proposition}

\theoremstyle{definition}
\newtheorem{defn}[thm]{Definition}

\newtheorem{remark}[thm]{Remark}
\newtheorem{notation}[thm]{Notation}

\newcommand{\Irr}{\operatorname{Irr}}

\newcommand{\up}[1]{\textup{#1}}

\newcommand{\Th}{\operatorname{Th}}

 \newcommand{\HSI}{\textsf{HSI}}
 \newcommand{\RSI}{\textsf{RSI}}

 \newcommand{\mtab}[4]{%
  \begingroup
  \renewcommand{\arraystretch}{1.05}%
  \begin{array}{@{}c@{\mkern2mu}c@{}}
    {\scriptstyle #1} & {\scriptstyle #2}\\[-0.9ex]
    {\scriptstyle #3} & {\scriptstyle #4}
  \end{array}%
  \endgroup
}

\newcommand{\triple}[4]{%
  \begin{gathered}
    #2\mkern7mu #3\mkern10mu #4\\[-0.15ex]
    \scriptstyle M_{#1}
  \end{gathered}%
}

\newcommand{\Taaaa}{\mtab{a}{a}{a}{a}}
\newcommand{\Taaab}{\mtab{a}{a}{a}{b}}
\newcommand{\Taabb}{\mtab{a}{a}{b}{b}}
\newcommand{\Tbbbb}{\mtab{b}{b}{b}{b}}
\newcommand{\Tabbb}{\mtab{a}{b}{b}{b}}
\newcommand{\Taaba}{\mtab{a}{a}{b}{a}}
\newcommand{\Tabba}{\mtab{a}{b}{b}{a}}

\begin{document}

\title[]{Multiplicatively idempotent HSI algebras satisfy all equations of $\mathbb{N}$}
\author{Tumadhir Alsulami}
\address{Department of Mathematics, Umm Al-Qura University, Makkah, Saudi Arabia} \email{tfsolami@uqu.edu.sa}
\author{Marcel Jackson}
\address{Department of Mathematical and Physical Sciences, La Trobe University, Victoria  3086,
Australia} \email{m.g.Jackson@latrobe.edu.au}
\author{Michael Kinyon}
\address{College of Natural Sciences and Mathematics, University of Denver} \email{michael.kinyon@du.edu}


\thanks{This work was completed while the first author was a PhD student at La Trobe University, and many of the results in the paper appear in her PhD thesis~\cite{als:phd}}

\begin{abstract}
An algebra with binary operations $+,\cdot,\uparrow$ and constant $1$ is called an HSI algebra if it satisfies the basic commutative semiring laws for $+,\cdot,1$ on~$\mathbb{N}$ as well as the familiar index laws for exponentiation $\uparrow$.  
These basic axioms, known as the ``High School Identities'' are known to be incomplete, and an algebra satisfying $\HSI$ but failing an equation valid on ${\bf N}:=\langle \mathbb{N};+,\cdot,\uparrow,1\rangle$ is called a \emph{Gurevi\v{c} algebra}.

It is currently unknown if there is an algorithm to recognise finite Gurevi\v{c} algebras, and the best current result is that there exists a 12-element Gurevi\v{c} algebra, and that none of the five 2-element HSI algebras are Gurevi\v{c} algebras.  
We explain how the work of Alex Wilkie can be used to provide an algorithm for deciding validity of the equational laws of $\langle \mathbb{N};+,\cdot,\uparrow,1\rangle$, and use this to show that multiplicatively-idempotent HSI algebras satisfy all valid laws of $\langle \mathbb{N};+,\cdot,\uparrow,1\rangle$.
As consequences of this result, we show that all 44 HSI algebras on 3 elements lie in the variety of $\langle \mathbb{N};+,\cdot,\uparrow,1\rangle$ (that is, are not Gurevi\v{c} algebras), as well as 597 of the 657 models on 4 elements and 11158 of the 13577 models on 5 elements.  
A further consequence is that every Brouwerian lattice lies in the variety of $\langle \mathbb{N};+,\cdot,\uparrow,1\rangle$ (up to a simple term equivalence), showing that the variety of $\langle \mathbb{N};+,\cdot,\uparrow,1\rangle$ has continuum many subvarieties.
The constant-free signature is also explored, and it is shown that in this case all $2$-element models of the constant-free High School Laws satisfy all valid constant-free laws in $\langle \mathbb{N};+,\cdot,\uparrow\rangle$.
\end{abstract}

\maketitle

\section{Introduction}\label{sec:intro}
\maketitle

Let $\uparrow$ denote the binary operation of exponentiation between
numbers and let
${\bf N}$ be the algebra $\langle\mathbb{N};+,\cdot,\uparrow,1\rangle$.  The familiar `high
school identities' of this structure can
be summarised as follows:
\[
{\rm HSI}\left\{\begin{array}{ccc}
&\left.\begin{array}{cc}
x+y \approx y+x,\hfill&
x+(y+z)
\approx (x+y)+z,\hfill\\
x \cdot y \approx y \cdot x, \hfill &x \cdot (y \cdot z) \approx (x \cdot y) 
	\cdot z,\hfill\\
x \cdot (y+z) \approx (x \cdot y)+(x
\cdot z),\hfill& x^{y+z} \approx x^y \cdot x^z,\hfill\\
(x \cdot y)^z \approx {x^z} \cdot {y^z},\hfill&
(x^y)^z \approx x^{y \cdot z},\hfill
\end{array}\right\}\RSI\vspace{.2cm}\\
&\hspace{1.5cm} x \cdot 1
\approx x,\
\
	1^x \approx 1,\  \ x^1 \approx x.\hfill &
\end{array}\right.
\]
For convenience we denote these identities by
${\rm HSI}$ but we will also explore the restricted collection of identities not involving the
constant $1$, which we denote by~$\RSI$ and use ${\bf N}^-$ for $\langle\mathbb{N};+,\cdot,\uparrow\rangle$.
We also let $\overline{\HSI}$ denote the exponentiation-free laws: the first two and half rows, with $x\cdot 1\approx x$; we use $\overline{\bf N}$ for $\langle\mathbb{N};+,\cdot,1\rangle$.

Alfred Tarski asked whether $\HSI$ is sufficient to derive all true laws of positive numbers.
Significant early contributions were made by his PhD student Charles Martin~\cite{mar}, who showed (amongst other interesting results) that the constant-free signature $\{+,\cdot,\uparrow\}$ has no finite equational basis; in particular $\RSI$ is insufficient.  
Tarski's question was resolved by Alex Wilkie~\cite{wil00} in 1980 who provided a negative solution to Tarski's question, by showing that the following law is valid, but not derivable:
\begin{multline*}
((1+x)^x+(1+x+x^2)^x)^y((1+x^3)^y+(1+x^2+x^4)^y)^x \\
\approx ((1+x)^y+(1+x+x^2)^y)^x((1+x^3)^x+(1+x^2+x^4)^x)^y.\tag{$*$}\label{eq:Wilkie}
\end{multline*}
(Wilkie's manuscript was not published until 2000, but had some distribution in manuscript form from 1980 onwards.)
The validity of Wilkie's law~\eqref{eq:Wilkie} is easily seen because both sides reduce to 
\[
(1-x+x^2)^{xy}((1+x)^x+(1+x+x^2)^x)^y((1+x)^y+(1+x+x^2)^y)^x\tag{$*$}\label{eq:WilkieLaw}
\]
once it is observed that $(1+x)(1-x+x^2)\approx 1+x^3$ and $(1+x+x^2)(1-x+x^2)\approx 1+x^2+x^4$ are valid over $\mathbb{R}$.
This is not a true equational deduction, as we allowed subtraction in the term $1-x+x^2$, yet $\mathbb{N}$ is not closed under subtraction (and moving to a setting such as $\mathbb{R}$ or $\mathbb{C}$ results in complications for exponentiation that shift from the simple appeal of Tarski's question: exponentiation is only a partial operation, for example).  
Later, R.~Gurevi\u{c}~\cite{gur90} used generalisations of Wilkie's law to show that no finite set of valid equations was sufficient to prove all valid equations, nor even all valid 1-variable equations.
A rather difficult second contribution of Wilkie's work in~\cite{wil00} is that it is only ever the polynomial factorisation issue that stands between completeness of the \HSI\ laws.
We give a more detailed account of this deep work in Section~\ref{sec:Wilkie}, as the result is an extremely powerful tool in the central proofs of the paper.
 Wilkie's law~\eqref{eq:WilkieLaw} is of course a somewhat arbitrary, though relatively simple, example of a nonderivable law.  
Even keeping to the same simple pattern, if $P,Q,R,S$ are polynomials with $P$, $Q$ irreducible over $\mathbb{Z}[X]$ where $\{PF,QF\}=\{R,S\}$ for some polynomial $F$, then we always have 
\[
(P^x+Q^x)^y(R^y+S^y)^x\approx (P^y+Q^y)^x(R^x+S^x)^y
\]
valid, and it will not be derivable if $R$ and $S$ do not have a common factor.  
A skew variant does not seem to have been noted previously:
 \[
(P^x+R^x)^y(Q^y+S^y)^x\approx (P^y+R^y)^x(Q^x+S^x)^y.
\]
Validity holds because, using the common factor $F$, both sides reduce to $P^{xy}Q^{xy}(1+F^x)^y(1+F^y)^x$.

Wilkie's proof of the nonderivability of his law~\eqref{eq:Wilkie} was partially proof theoretic, but an elegant alternative approach was pioneered by Gurevi\u{c}~\cite{gur85} who showed that any valid non-consequence $s\approx t$ of \HSI\ must have a finite counterexample: a finite model of $\HSI$ that fails $s\approx t$.  
This gives rise to the obvious question as to whether there is a small counterexample to some valid law.
Following our earlier paper~\cite{alsjac2} we say that an algebra is a \emph{Gurevi\u{c} algebra} if it satisfies $\HSI$ but fails some valid law.  (In Burris and Lee~\cite{burlee92} the same name was reserved for counterexamples to Wilkie's law specifically; Gurevi\u{c} himself considered counterexamples to multiple laws, including in his proof of the nonfinite basis property, and the relatively arbitrary nature of Wilkie's original example makes the more general notion we use here seem more sensible.)
The \emph{exponential algebra problem}~\cite{alsjac2} asks of a given finite HSI algebra whether it is a Gurevi\u{c} algebra.  
This problem seems difficult to resolve even for a typical small algebra, and the decidability is currently open.
Burris and Lee~\cite{burlee92} enumerated the 2- and 3-element HSI algebras, noting that there are just three 2-element examples that are not quotients of ${\bf N}$.  
The question as to whether any of these three is a Gurevi\u{c} algebra is already quite nontrivial and was left as an open problem.  This was solved by Asatryan~\cite{asa} and the second author (unpublished, though presented in several talks).
Given the challenges of the 2-element case, the authors initially felt that it would be unrealistic to expect a complete solution for the $3$-element case, where 42 of the 44 HSI algebras on 3 elements are nontrivial (there are two quotients of ${\bf N}$, so these are not Gurevi\u{c} algebras).  
In the present paper however, we present the following general result that covers all but one of the 42 nontrivial cases, and dispatch the 42nd case by ad hoc means.
\begin{thm}\label{thm:main}
$\HSI\cup\{x^2\approx x\}$ is sufficient to derive all true laws of ${\bf N}$.
\end{thm}
We also consider the constant-free signature explored by Martin.  Here there are 18 distinct 2-element examples, and we are able to piggyback on  Theorem~\ref{thm:main} to show that all of these also lie in the variety of ${\bf N}$, though there is a 6-element counterexample to the simplest of Martin's nonderivable laws.

In terms of an upper bound for the smallest Gurevi\u{c} algebra, after an initial 56-element example was found by Gurevi\u{c}~\cite{gur85}, a number of authors provided gradually smaller examples~\cite{burlee92,gur90,jac96} until a $12$-element example of Burris and Yeats~\cite{buryea} halted progress for over 20 years.  
A computational search by Zhang~\cite{zha} had shown that at least 11-elements is necessary, and more recent efforts by the authors using the counterexample finder Mace4 and utilising approaches explained by Burris and Lee~\cite{burlee92} and other techniques (which were also used by Zhang) seemed to suggest that there was little prospect of further progress (a range of similar laws were explored).
Two recent SAT-solver approaches seem to have reached a new level of capability in such searches.  Independently, Hajdari and Niederhauser~\cite{hajnie} and Subercaseaux and Przybocki~\cite{subprz} have computationally verified that there are no smaller counterexamples to Wilkie's identity than the one found by Burris and Yeats, as well as finding many further 12-element examples. 
Remarkably, the authors of~\cite{subprz} were able to achieve a simple parameterisation of all possible 12-element counterexamples to Wilkie's law, of which there are 8,957,952.  
The authors of~\cite{subprz} also considered some other variants of  Wilkie's law (including some suggested by the authors of the present paper---the skew variants mentioned above for example), but the smallest counterexample found for any other law so far has 13-elements.

\subsection{Structure of paper}
We begin in Section~\ref{sec:Wilkie} by presenting the details for Wilkie's positive contribution to Tarski's $\HSI$ problem: the proof that polynomial factorisation is the only hindrance to completeness of $\HSI$.  
Thus a very careful high school student can in principle deduce all valid exponential laws using only high school methods, given enough time.  
The main purpose is to explain the key required elements for use in subsequent sections, though we take care to demonstrate that the process is effective, and requires nothing beyond factorisation into irreducible polynomials over $\mathbb{Z}$.  
This includes factorisations of numbers, so the process is not polynomial time as far as we can see, even though it is conceptually simple.
We mention that Macintyre~\cite{mac} (after Richardson~\cite{ric}) already showed the decidability of the equational theory of ${\bf N}$ in 1981, with a further algorithm given by Gurevi\u{c}~\cite{gur85}; these algorithms are fundamentally different in character to the one we give here, essentially showing that the number of points of agreement of two exponential terms can be algorithmically bounded.

In Section~\ref{sec:idempotent} we use Wilkie's result to prove Theorem~\ref{thm:main}: idempotent HSI algebras are not Gurevi\u{c} algebras. 
The $3$-element case is then completed: there are no 3-element Gurevi\u{c} algebras.
The use of Wilkie's result in this context was previously invoked by both Asatryan~\cite{asa,asa2} and the second author in the proof of the $2$-element case as well.
We provide a further application of Wilkie's result in Section~\ref{sec:LEA}, in order to create a basic infinite HSI algebra in the variety of ${\bf N}$ that has many useful small quotients.

In Section~\ref{sec:Martin} we consider the constant-free signature, using the results of both Section~\ref{sec:idempotent} and Section~\ref{sec:LEA} to show that all $2$-element models of $\RSI$ are not Gurevi\u{c} algebras (in the constant free sense).  We also provide a $6$-element counterexample, relative to $\RSI$, for several small laws.

\section{Wilkie's Normal Form}\label{sec:Wilkie}
The original (negative) resolution of Tarski's High School Algebra problem is to be found in Alex Wilkie~\cite{wil00}, which in hand-written form was distributed as early as 1980.  
The bulk of that paper is devoted to showing a positive result, which could be summarised as saying that familiar High School style algebra is in fact sufficient to derive all true laws of ${\bf N}$, once enough information on polynomial factorisation is allowed: we need to be able to factorise polynomials in the semiring $\mathbb{N}$.  
(This is somewhat reminiscent of the temporary emergence of imaginary parts in the usual algorithm for solving the cubic and quartic, even in the case when the solutions are ultimately real numbers.)
Of course, $\mathbb{N}$ is not itself closed under subtraction, so care must be taken to formalise the approach.  
Once this is performed however, the argument provides a routine reduction procedure to a kind of normal form: the difficulty in Wilkie's argument is showing that certain exponential expressions are algebraically independent.
We now give an overview of Wilkie's approach, as needed for application in the main result. 

A \emph{polynomial} will be assumed to have coefficients from $\mathbb{Z}$ and indeed can be thought of as elements of the polynomial ring $\mathbb{Z}[y_1,y_2,\dots,x_1,x_2,\dots]$, though sometimes they will be considered as elements of $\mathbb{R}[y_1,y_2,\dots,x_1,x_2,\dots]$ (with coefficients from $\mathbb{Z}$), sometimes as real-valued functions, and sometimes as elements of the polynomial semiring $\mathbb{N}[y_1,y_2,\dots,x_1,x_2,\dots]$, if all of the coefficients happen to be nonnegative.  
A  polynomial $\rho$ is \emph{positive} if $\rho(r_1, r_2, \dots , r_n)\in \mathbb{R}^+$ when $r_1, r_2, \dots , r_n\in  \mathbb{R^+}$ and is \emph{strictly positive} if all of its nonzero coefficients are from $\mathbb{N}$, that is, if it can be considered as a member of $\mathbb{N}[y_1,\dots,x_1,\dots]$.  
The set of positive polynomials in $k$-variables will be denoted by~$\mathcal{P}_k$, though the notation $\mathcal{P}_{m+n}$, where $m+n=k$ is often used to mean that there are $m$ variables of type $y$ and $n$ of type $x$: this is a bookkeeping device used by Wilkie, because later the $x$ variables will be substituted by other terms, whereas the $y$ variables will always remain as individual variables.
The members of~$\mathcal{P}_k$ that are irreducible over $\mathbb{Z}$ and are neither $1$, nor a single variable, will be denoted by $\Irr_k$.  
The two variable kinds are treated differently, but the $y$-variables will eventually be treated the same as for members of~$\Irr_{k}$.  
The exclusion of the single variable case from $\Irr_{k}$ is again because of the role of the two different kinds of variables.
We use $\mathcal{P}_k^+$ and $\Irr_k^+$ to denote  the strictly positive members of $\mathcal{P}_k$ and $\Irr_k$, respectively.  
The members of~$\mathcal{P}_k$ that are monomial with coefficient~$1$ (but not $1$ itself), will be denoted by $\mathcal{M}_k$: in other words the elements of $\mathcal{M}_k$ are simply the $\{\cdot\}$-terms with variables from $y_1,y_2,\dots, y_m,x_1,\dots,x_n$ (with $m+n=k$).
We use the same notations $\mathcal{P},\Irr,\mathcal{M}$ without subscript when the precise number of variables is free.

\begin{notation} The convention of listing $y$-variables before  $x$-variables in tuples follows Wilkie's paper~\cite{wil00}, and this notation has been kept in order to not create confusion when consulting that work.   
We have deviated slightly from Wilkie's notation in the notation $\mathcal{P},\mathcal{P}^+, \Irr,\Irr^+,\mathcal{M}$ to aid readability, as Wilkie moves between the $\mathcal{P}$ and $\Irr$ families by varying the font for $P$ only.
\end{notation}

For each positive polynomial $\rho(y_1,\dots,y_m,x_1,\dots,x_n)$ of~$\mathcal{P}_{m+n}$, introduce a new $(m+n)$-ary operation symbol $t_\rho$, defined on $\mathbb{R}^+$ (or $\mathbb{N}$) by $t_\rho(b_1,\dots,b_m,a_1,\dots,a_n):=\rho(b_1,\dots,b_m,a_1,\dots,a_n)$.  
We let $\mathscr{L}$ denote the usual signature $\{+,\cdot,\uparrow,1\}$ and let the infinite signature $\{+,\cdot,\uparrow,1\}\cup\{t_\rho\mid \rho\in \mathcal{P}\}$ be denoted by $\mathscr{L}^*$.  
Next, let $\HSI^*$ denote the expansion of the equational laws $\HSI$ to include all laws of the form $t_{\rho_1}\dots t_{\rho_n}\approx t_\rho$, where~$\rho$ is strictly positive and decomposes as a product $\rho_1\cdot\dots \cdot \rho_n$ of irreducible positive polynomials, as well as those laws that translate strictly positive polynomial operations to their $\{+,\cdot,1\}$ counterpart: $t_\rho\approx \rho$, whenever $\rho\in\mathcal{P}^+$ (recalling that polynomials in $\mathcal{P}^+$ are also $\{+,\cdot,1\}$ terms). 
We note that Wilkie considers the larger set of equations consisting of all laws
$s\approx t$, where $s$ and $t$ are $\mathscr{L}^*$-terms not involving exponentiation and such that $\mathbb{N}\models s\approx t$.  
Only laws from $\HSI^*$, however, are used in the proof, as can be checked in our exposition below (proof of Lemma~\ref{lem:WilkieNormalForm}).

Despite the flagship negative resolution of Tarski's HSI problem, the following result is arguably the principal result of Wilkie's paper.
\begin{WilkiesTheorem}\label{thm:WilkiesTheorem}
The expanded set of laws $\HSI^*$ is complete for the equational theory of $\langle \mathbb{N},+,\cdot,\uparrow,1\rangle$.
\end{WilkiesTheorem}
The proof of Wilkie's HSI Theorem~\ref{thm:WilkiesTheorem} has two parts: the reduction to a kind of normal form using $\HSI^*$; then the proof that certain terms used in these normal forms are algebraically independent.
The second fact implies that the normal forms are unique, so that $\HSI^*$ proves all valid laws in the signature $\mathscr{L}$.
  
We give details of the first part, verifying that $\HSI^*$ (rather than the larger set used by Wilkie) is sufficient.  The following fact is an easy consequence of Wilkie's HSI Theorem~\ref{thm:WilkiesTheorem}, but otherwise not obviously true.  Here $\uparrow_i$ denotes exponentiation with $0^0:=i$.
\begin{remark}\label{eg:N0}
Each of the algebras $\langle \mathbb{N},+,\cdot,\uparrow,1\rangle$, $\langle \mathbb{N}_0,+,\cdot,\uparrow_0,1\rangle$ and $\langle \mathbb{N}_0,+,\cdot,\uparrow_1,1\rangle$ satisfy the same equations.
\end{remark}
\begin{proof}
As $\langle \mathbb{N},+,\cdot,\uparrow,1\rangle$ is a subalgebra of the other two algebras,  no extra laws are satisfied by the model $\langle \mathbb{N}_0,+,\cdot,\uparrow_0,1\rangle$ and $\langle \mathbb{N},+,\cdot,\uparrow_1,1\rangle$ in comparison to $\langle \mathbb{N},+,\cdot,\uparrow,1\rangle$.  To show that the larger models satisfy each equation of $\langle \mathbb{N},+,\cdot,\uparrow,1\rangle$, first observe that both
 satisfy $\HSI$: it is required only to check those cases where~$0^0$ arises in evaluating the terms, and these are trivially verified.  Second, all polynomial equations (between positive polynomials) true on $\mathbb{R}^+$ hold true on all of~$\mathbb{R}$, so on $\mathbb{R}_0^+$ in particular.  So $\langle \mathbb{N}_0,+,\cdot,\uparrow_0,1\rangle$ and $\langle \mathbb{N},+,\cdot,\uparrow_1,1\rangle$ can be made into  $\HSI^*$-algebras.
 From Wilkie's Theorem~\ref{thm:WilkiesTheorem}, they satisfy all true laws of $\langle \mathbb{N},+,\cdot,\uparrow,1\rangle$.
\end{proof}
This result fails if $0$ is included in the signature because $\langle \mathbb{N}_0,+,\cdot,\uparrow_0,1\rangle\models 0^x\approx 0$, but not $x^0\approx 1$, while the reverse true for $\langle \mathbb{N}_0,+,\cdot,\uparrow_1,1\rangle$.

A second easy consequence is the following.
\begin{remark}\label{rem:HSIproof}
If $s\approx t$ is a valid law of $\langle \mathbb{N},+,\cdot,\uparrow,1\rangle$, and there is a  proof of $s\approx t$ from the expanded laws $\HSI^*$ involving only operations from $\{+,\cdot,\uparrow,1\}\cup\{t_\rho\mid \rho\in \mathcal{P}^+\}$, then $\HSI\vdash s\approx t$.
\end{remark}
\begin{proof}
At each step of the proof from $\HSI^*$, use $\rho$ in place of $t_\rho$, noting that $\rho\in \mathcal{P}^+$ is an $\mathscr{L}$-term.
\end{proof}

The normal form is defined relative to some infinite sequence of pairs of polynomials $(p_1,q_1),(p_2,q_2), \dots (p_k,q_k), \dots, (k\in \mathbb{N})$ in  $\mathscr{L}^*$, satisfying the following conditions (W1.1), (W1.2), (W2.1), (W2.2).
For each $k=0,1,\dots$, we require:
\begin{itemize}
\item[(W1.1)] $p_{k+1}= t_{\rho}( y_1, y_2, \dots , y_n, x_1, x_2, \dots , x_{k'})$ for some $\rho\in \Irr_{n+k'}$ and $k'\leq k$, or $p_{k+1}=y_i$ for some $i\in \mathbb{N}$ and $i\leq n$.
\item[(W1.2)] $q_{k+1}= t_{\mu}( y_1, y_2, \dots , y_n, x_1, x_2, \dots , x_{k''})$ 
for some $\mu \in \mathcal{M}_{n+k''}$ and $k''\leq k$.
\item[(W1.3)] For distinct $j,k$, the identities $p_j\approx p_k$ and $q_j\approx q_k$ do not both hold on $\mathbb{N}$.
\end{itemize}
We also require the following:
\begin{itemize}
\item[(W2.1)] For each $k,k'\in \mathbb{Z^+}$ and for each $\rho\in \Irr_{n+k}$ and $\mu \in \mathcal{M}_{n+k'}$ there is some $\ell\in \mathbb{N}$ such that 
\[
\HSI^* \vdash p_\ell\approx  t_{\rho( y_1, y_2, \dots , y_n, x_1, x_2, \dots , x_k)}
\] 
and 
\[
\HSI^* \vdash q_\ell\approx t_{\mu}( y_1, y_2, \dots , y_n, x_1, x_2, \dots , x_{k'}).
\]
\item[(W2.2)] For each $k\in \mathbb{N}$ and $i\leq n$, and $\mu\in \mathcal{M}_{n+k}$ there is some $\ell\in \mathbb{N}$ such that: $\HSI^* \vdash p_\ell\approx y_i$ and $\HSI^*\vdash q_\ell \approx t_\mu$.
\end{itemize}
While these seem somewhat technical, it is apparent that there is enormous freedom in the construction of the sequence.  The second phase of the proof of Wilkie's HSI Theorem~\ref{thm:WilkiesTheorem} needs such a sequence fixed, but the first phase of the proof (reduction to normal form) does not make significant use, only requiring that polynomials encountered during an application of the reduction process can be placed within a larger list conforming to (W1.1), (W1.2), (W2.1), (W2.2).  It is clear that any sequence of distinct irreducible polynomials and monomials can be extended, in principle, to such a list, so that the second phase of Wilkie's proof will ensure that the reduction process is guaranteed to lead to the same normal form when initiated on two terms $s,t$ where $s\approx t$ is a valid law.
We summarise this and some further algorithmic observations in the following lemma.
\begin{lem}\label{lem:listextend}
For any finite set $\mathcal{F}\subseteq\mathcal{P}^+$ of strictly positive polynomials, there is an algorithm that produces all factors of members of $\mathcal{F}$ in $\Irr$, all monomials appearing as terms within polynomials in $\mathcal{F}$ and creates a list $(p_1,q_1),\dots,(p_n,q_n)$ for some $n$, that can be in-principle extended to a full enumeration satisfying \up(W1.1\up)--\up(W2.2\up).
\end{lem}
\begin{proof}
In the polynomial ring $\mathbb{Z}[X]$, the factorisation into  factors that are irreducible over~$\mathbb{Z}$ is effectively computable.  Moreover for fixed $X$ it is ostensibly polynomial time via Kaltofen's algorithm~\cite{kal} and the Lenstra-Lenstra-Lovasz algorithm~\cite{LLL} in the case of primitive polynomials, however we need full factorisation including any constant coefficients, and factorisation of integers has no known polynomial time solution currently.  
All factors of all polynomials and subpolynomials from $\mathcal{F}$ can be factorised using this algorithm and added to $\mathcal{F}$.  
These polynomials, plus all relevant variables from $y_1,\dots$ that appear in polynomials from $\mathcal{F}$ can be placed into the initial segment of the sequence $(p_1,q_1),(p_2,q_2),\dots$, possibly using some extra polynomials $x_1,x_1x_2,x_1x_2x_3,\dots$ to fill the initial terms, in order that the technical conditions $k'\leq k$ and $k''\leq k$ are not violated in (W1.1) and (W1.2).  
The resulting  finite sequence $(p_1,q_1)$, $(p_2,q_2)$, \dots, $(p_n,q_n)$ can always be extended to a full list satisfying (W1.1)--(W2.2):  the requirements are only to avoid repeats (up to equivalent polynomials) and to ensure all combinations of the required format of the $p$ terms and the $q$ are listed.  
\end{proof}
We note that in order to obtain the proof $\HSI^*\vdash s\approx t$,  we will not be required to replicate all of the steps in the proof just given.  The lemma is just to guarantee that the second phase of the proof of Wilkie's Theorem will hold (for some infinite listing of pairs), thereby implying that the normal reduction will be unique.

Next Wilkie defines a sequence of exponential terms $\tau_1,\tau_2,\dots$, by way of some intermediate terms $u_1,u_2,\dots,$ and $s_1,s_2,\dots$
\begin{itemize}
\item[(W3.1)] $u_1:=p_1$, $s_1=q_1$, $\tau_1=u_1^{s_1}$
\item[(W3.2)]  $u_{i+1}:=p_{i+1}(y_1,\dots,y_n,\tau_1,\dots,\tau_i)$ and $s_{i+1}:=q_{i+1}(y_1,\dots,y_n,\tau_1,\dots,\tau_i)$, and then $\tau_{i+1}:=u_{i+1}^{s_{i+1}}$.
\end{itemize}
Note that in (W3.2), the variables in $p_{i+1}$ are $y_1,\dots,y_n,x_1,\dots,x_{k}$ for some $k\leq  i$, and the notation $p_{i+1}(y_1,\dots,y_n,\tau_1,\dots,\tau_i)$ means that $\tau_1,\dots,\tau_k$ have been substituted for $x_1,\dots,x_k$ (while $x_{k+1},\dots,x_i$ do not appear in $p_{i+1}$, meaning that $\tau_{k+1},\dots,\tau_i$ are listed only to make for a canonical exposition of the term).

The polynomial $\rho_f\in \mathcal{P}_{n+m}^+$ in the following lemma will be said to \emph{represent} the term $f$ and provides a normal form for $\mathscr{L}$-terms, relative to the ordering  $(p_1,q_1)$, $(p_2,q_2)$,\dots\ and $\tau_1$, $\tau_2$, \dots.
\begin{lem}\label{lem:WilkieNormalForm}
If $f$ is an $\mathscr{L}$-term in variables $y_1,y_2, \dots ,y_n$, then there is a recursive process that produces a strictly positive polynomial $\rho_f \in \mathcal{P}_{n+m}^+$ for some $m \in \mathbb{Z^+}$  such that
\[
\HSI^* \vdash  f\approx t_{\rho_f}   (y_1, \dots , y_n,\tau_1, \tau_2, \dots, \tau_m).
\]
\end{lem}
\begin{proof}
We repeat the proof of Theorem~2.8 of~\cite{wil00} but expanded in detail to demonstrate that each stage can be made automatic and that only laws from $\HSI^*$ are used. 
The proof is by induction on $f$, and we will also build into this reduction the property that representing polynomials are given as a sum of monomials.
The base case when $f=1$ or a single variable $y_i$ is trivially satisfied: the strictly positive polynomial $\rho_f \in \mathcal{P}_n$ can be chosen as $1$ or $y_i$, respectively.  Now consider a term $f\mathbin{\Box} g$, where $\Box\in\{+,\cdot,\uparrow\}$.  By the induction hypothesis, there are $\rho_f , \rho_g$ representing $f $ and $g $, respectively.  

Case 1.  Representing $(f+g)$ and $(f\cdot g)$.\\
In this case the polynomials $\rho_{f+g}:=\rho_f+ \rho_g$ and $\rho_{f\cdot g}:=\rho_f \cdot \rho_g$ are strictly positive and represent  $(f+g)$ and $(f\cdot g)$, respectively.  
Note that the case of $\rho_{f\cdot g}$ will require expansion of the product $\rho_f \cdot \rho_g$ in order to obtain a sum of monomials

Case 2.  Representing $f^g$.\\
If $f$ is  $1$, then $f^g$ can be represented by the strictly positive polynomial $1$, using the $\HSI$ law $1^x\approx 1$.  Now assume that $f$ is not $1$.  The positive polynomial $\rho_g$ representing $g$ is a sum of monomials:
\[
\rho_g = c_0+\sum_{i=1}^\ell c_i\mu_i,
\] 
where $\ell\geq 0$, the $c_i$ are positive integer constants, except for $c_0$ which is a non-negative integer constant, and the $\mu_i$ are monomials with coefficient $1$.  By the induction hypothesis, the term $f^g$ is equivalent under $\HSI^*$ to the replacement in~$t_{\rho_f}^{t_{\rho_g}}$ of the variables $x_1,x_2,\dots$ by the expressions $\tau_1,\tau_2,\dots$.  Using left to right applications of $x^{y+z}\approx x^yx^z$, 
we have 
\begin{equation}
t_{\rho_f}^{t_{\rho_g}} \approx {}\stackrel{c_0}{\overbrace{t_{\rho_f}^1\cdot t_{\rho_f}^1\cdot\cdots\cdot t_{\rho_f}^1}}\cdot\prod_{i=1}^{\ell}(\stackrel{c_i}{\overbrace{t_{\rho_f}\cdot t_{\rho_f}\cdot\cdots \cdot t_{\rho_f}}})^{\mu_i}.\label{eqn:wilkie1}
\end{equation}
Then using $x^1\approx x$ and left to right applications of $(xy)^z\approx x^zy^z$ we find that $t_{\rho_f}^{t_{\rho_g}}$ is equal to 
\begin{equation}
\stackrel{c_0}{\overbrace{t_{\rho_f}\cdot t_{\rho_f}\cdot\cdots\cdot t_{\rho_f}}}\cdot \prod_{i=1}^{\ell}(\stackrel{c_i}{\overbrace{t_{\rho_f}^{\mu_i}\cdot t_{\rho_f}^{\mu_i}\cdot\cdots\cdot t_{\rho_f}^{\mu_i}}}).\label{eqn:wilkie2}
\end{equation}
Note that each $\mu_i$ is in $\mathcal{M}_{n+k'}$ for some $k'$ possibly depending on $i$.
Next, the polynomial $\rho_f$ can be factorised as a product of irreducible (over $\mathbb{Z}$) factors as 
\[
\rho_f=\rho_1\cdot\rho_2\cdot \cdots\cdot \rho_s,
\]
where for each $i$ there is $m_{k'}$ such that $\rho_i\in \Irr_{n+m_{k'}}$, with $k'$ possibly depending on~$i$.  
As noted earlier, the factorisation of rational polynomials into irreducible factors can be performed in polynomial time, using the LLL algorithm of Lenstra, Lenstra and Lovasz~\cite{LLL} (this is up to a constant factor: factorisation of integers is of course not known to be solvable in polynomial time, though it can be solved recursively).  
This achieves a factorisation into integer coefficient polynomials, because rational factors of integer polynomials have integer coefficients (up to a rational factor), by Gauss' Irreducibility Lemma (see \cite[Theorem III.15]{birmac} or any standard algebra text).

Now each of the terms $t_{\rho_f}^{\mu_i}$ in the product \eqref{eqn:wilkie2} can be written using left to right applications of $(xy)^z\approx x^zy^z$ as $t_{\rho_1}^{\mu_i}\cdot t_{\rho_2}^{\mu_i}\cdot \cdots\cdot t_{\rho_s}^{\mu_i}$.  
Next we consider each $\rho_j^{\mu_i}$ individually, finding a representing polynomial $\alpha$ in each case; then the inductive steps in Case 1, or further application of Case 2 in subterms can be used to combine them.  

Subcase 2.1: $\rho_j=1$.\\
This can only happen if $f=1$, in which case the representing polynomial for $f^g$ is~$1$ (using $1^x\approx 1$) as was considered already.  

Subcase 2.2.  $\rho_j\in \Irr_{n+k'}$ or is a single variable $y_{i'}$.\\
In this case, conditions (W2.1), (W2.2), guarantee that there is an $i''$ such that $\rho_j=p_{i''}$ and $\mu_i=q_{i''}$ and $t_{\rho_i}(y_1,\dots,y_n,\tau_1,\dots)^{t_{\mu_i}(y_1,\dots,y_n,\tau_1,\dots)}$ is equal to $\tau_{i''}$ for some $i''$.  The representing polynomial can be chosen as $x_{i''}$.  

Subcase 2.3.  $\rho_j$ is a single variable $x_{i'}$.\\
In this case we have that $\rho_j(y_1,\dots,y_n,\tau_1,\dots,\tau_{i'},\dots)^{\mu_j(y_1,\dots,y_n,\tau_1,\dots,\tau_{i'},\dots)}$ 
is simply $\tau_{i'}^{\mu_j(y_1,\dots,y_n,\tau_1,\dots,\tau_{i'},\dots)}$.  
But $\tau_{i'}$ is of the form $\rho(y_1,\dots,y_n,\tau_1,\dots)^{\mu(y_1,\dots,y_n,\tau_1,\dots)}$, so that we can use a left to right application of $(x^y)^z\approx x^{yz}$ to obtain 
\[
\rho(y_1,\dots,y_n,\tau_1,\dots)^{\mu(y_1,\dots,y_n,\tau_1,\dots)\cdot \mu_j(y_1,\dots,y_n,\tau_1,\dots)}.
\]  
Because $\mu$ and $\mu_j$ are monomials, then so also is $\mu\mu_j$ a monomial $\mu'\in \mathcal{M}_{n+k''}$ for some $k''$ and we have arrived in the situation considered in Subcase 2.2.
\end{proof}

The fundamental role of representing polynomials is provided in~\cite[Theorem 4.4]{wil00}: the functions $y_1,\dots,y_n,\tau_1,\tau_2,\dots$ are algebraically independent over $\mathbb{R}^+$.  
Then, as explained in the proof of Theorem 1.9 of \cite{wil00} it follows that if $\langle \mathbb{N},+,\cdot,\uparrow,1\rangle\models f\approx g$, then Lemma~\ref{lem:WilkieNormalForm} shows that $\HSI^*\vdash f\approx t_{\rho_f}(y_1,\dots,y_n,\tau_1,\dots)$ and $\HSI^*\vdash g\approx t_{\rho_g}(y_1,\dots,y_n,\tau_1,\dots)$.  
So 
\[
\langle \mathbb{N},+,\cdot,\uparrow,1,\{t_{\rho}\mid \rho\in \mathcal{P}\}\rangle\models t_{\rho_f}(y_1,\dots,y_n,\tau_1,\dots)\approx t_{\rho_g}(y_1,\dots,y_n,\tau_1,\dots).
\]  
Then algebraic independence ensures that $\langle \mathbb{N},+,\cdot,\uparrow,1\rangle\models \rho_f\approx \rho_g$.  
That is, $f$ and $g$ have the same representing polynomial (up to application of $\overline{\HSI}$).
We summarise this, in the following theorem.
\begin{thm}\label{thm:samerep}
$\langle \mathbb{N},+,\cdot,\uparrow,1\rangle\models f\approx g$ if and only if $f$ and $g$ have the same representing polynomials.  
\end{thm} 
The algorithmic version of Wilkie's Theorem~\ref{thm:WilkiesTheorem} is an immediate corollary of Theorem~\ref{thm:samerep} and Lemma~\ref{lem:WilkieNormalForm}.
\begin{cor}
The equational theory of $\langle \mathbb{N},+,\cdot,\uparrow,1\rangle$ is decidable by way of the recursive process identified in Lemma \ref{lem:WilkieNormalForm}.
\end{cor}

\begin{remark}\label{rem:normalform}
Subject to the ordering of the sequence $(p_1,q_1),(p_2,q_2),\dots$, the representing polynomials provide a normal form representation of $\mathscr{L}$-terms in the expanded positive polynomial signature $\mathscr{L}^*=\{+,\cdot,\uparrow,1\}\cup \{t_{\rho}\mid \rho\in \mathcal{P}\}$, with two terms being equivalent if they reduce to the same normal form.  We refer to this as the \emph{Wilkie Normal Form} of an expression, noting that it does depend on some fixed enumeration of the special pairs $(p_1,q_1),(p_2,q_2),\dots$.  Aside from factorisation of polynomials and the $\overline{\HSI}$-equivalence of polynomials (Case 1), the reduction to normal form involves only the following steps.
\begin{itemize}
\item In Case 2, preparations for subcases: left to right applications of $x^{y+z}\mapsto x^yx^z$ (for Equation~\eqref{eqn:wilkie1}); left to right applications of $x^1\mapsto x$ and $x^{y+z}\mapsto x^yx^z$ (for Equation~\eqref{eqn:wilkie2}); after factorisation into irreducible factors, left to right applications of $(xy)^z\mapsto x^zy^z$.
\item Subcase~2.1: left to right applications of $1^x\mapsto 1$.
\item Subcase~2.3: left to right applications of $(x^y)^z\mapsto x^{yz}$.
\end{itemize}
Subcase~2.2  involves no actual step of deduction, only the observation that there exists a $\tau$ expression matching the current subcase. 
\end{remark}

We observe that it is now apparent that Wilkie's 1980 manuscript provides an alternative approach to several results obtained via different means in Henson and Rubel~\cite{henrub}.
\begin{thm}\label{thm:fbHR}
\begin{enumerate}
\item 
$\HSI$  is complete for equations involving only fixed base exponentiation: $\{ +,\cdot, 1\}\cup\{\exp_b\mid b\in \mathbb{N}\}$.
\item The $\HSI$ laws are complete for terms involving only  $\{ \cdot, \uparrow, 1,2,\dots\}$, subject to the definitions of $1+1\approx 2$, $1+1+1\approx 3$, and so on.  
\end{enumerate}
\end{thm}
\begin{proof}
Item (1) 
 is an application of Remark~\ref{rem:HSIproof}: simply verify that the process for identifying representing polynomials identified in the proof of Lemma~\ref{lem:WilkieNormalForm} never factors any polynomial into a positive but not strictly positive polynomial.  
 This is trivial for the base case and Case~1, where no factoring occurs.  
 For Case~2, observe that $f$ is a constant~$b$ (if only fixed base exponentiation occurs), and so the factors of $b$ are just the prime factors of $b$ counted with multiplicity, which remain strictly positive, and the equational properties of constants are provable by $\HSI$. 
 
 For item (2), let $s\approx t$ be a true identity in the signature $\{\cdot, \uparrow ,1,2,3,\dots\}$.  
 All representing polynomials in the proof of $s\approx t$ from $\HSI^*$ are monomials, and so by Remark \ref{rem:HSIproof} it follows that $s\approx t$ can be proved from $\HSI$.  
\end{proof}
We do not suggest that Wilkie's proof provides an easier route to these results, as the second phase of the proof of Wilkie's HSI Theorem~\ref{thm:WilkiesTheorem} is quite nontrivial.
Yet another approach to the two results in Theorem~\ref{thm:fbHR} can be found in the 1973 PhD thesis of Charles Martin~\cite{mar}.  
Let $T$ denote the set of all $\mathscr{L}$-terms, in which for any subterm of the form $s^t$, the expression $s$ contains no applications of $+$, except between constants.  
Clearly $\{ \cdot, \uparrow , 1,2,3,\dots\}$-terms and $\{ +,\cdot, 1\}\cup\{\exp_b\mid b\in \mathbb{N}\}$ terms can be placed within $T$ (subject to how constants are to be represented).  
Martin shows that $\HSI$ is complete for equations involving terms from $T$, a fact that can also be reproved using Remark~\ref{rem:HSIproof}.

\begin{thm} \up(Martin~\cite[Theorems 4 and 1.50]{mar}\up)
$\HSI$  is complete for equations involving terms from~$T$.
\end{thm}
\begin{proof}
This can again be obtained by Remark~\ref{rem:HSIproof}: the base cases and Case~1 are trivial as they avoid factorising, while in Case 2, for representing $f^g$, the term~$f$ avoids $+$ (except between constants) so is represented by a monomial, and its factorisation produces individual variables only.
\end{proof}

\section{Idempotent HSI algebras}\label{sec:idempotent}
The main result of this section is the following theorem (it is a trivial reformulation of Theorem~\ref{thm:main}), which can be equivalently phrased as saying that there are no multiplicatively idempotent Gurevi\v{c} algebras.
\begin{thm}\label{thm:x2x}
The variety of multiplicatively idempotent HSI algebras is a subvariety of the variety generated by ${\bf N}$.
\end{thm}
The proof of Theorem~\ref{thm:x2x} will cover most of the section and consist of a number of lemmas.  

Before we commence these we observe that a number of other strong laws imply multiplicative idempotence; these are not used in the proof, but are commonly observed properties of many very small HSI algebras. 
\begin{lem}
Each one of the following laws implies $x^2\approx x$ in the presence of \HSI, so also define subvarieties of the variety of ${\bf N}$\up: $x^y\approx x$, $x+y\approx x\cdot y$, $1\approx 2$, $x+x\approx x$.
\end{lem}
\begin{proof}
The consequence $x^y\approx x\vdash x^2\approx x$ is immediate, as is $2\approx 1\vdash x^2\approx x$.  From either of $x+x\approx x$ or $x+x\approx x\cdot x$ we obtain $2=1+1\approx 1$.
\end{proof}

The next lemma is used in the proof.
\begin{lem}\label{lem:x2x}
The following laws are consequences of $\HSI\cup\{x^2\approx x\}$\up:
\begin{align*}
x+y\approx {}&x+2xy +y,\\
1+x\approx {}& 1+3x,\\
x+xy\approx{}& x+3xy,\\
2\approx{}& 4.
\end{align*}
\end{lem}
\begin{proof}
Using $x^2\approx x$ we have $x+y\approx (x+y)^2\approx x^2+2xy+y^2\approx x+2xy+y$.  Then $1+x\approx 1+3x$ follows also when $y$ is $1$, with $x+xy\approx x+3xy$ following by multiplication of both sides by $x$.  The law $2\approx 4$ follows similarly by letting $x=1$.
\end{proof}

Let us say that a monomial $m$ is  \emph{covered} by some monomials~$m_1$, $m_2$, \dots, $m_k$ if $m$ has exactly the same explicitly appearing variables as the product $m_1m_2\dots m_k$.
\begin{lem} \label{lem:cover2}
Let $m$ be a monomial that is covered by some monomials~$m_1$, $m_2$, \dots, $m_k$ with $k>1$, then 
\[
\HSI\cup\{x^2\approx x\}\vdash  m_1+ \dots + m_k \approx  m_1+ \dots + m_k+ 2m
\]
\end{lem}
\begin{proof}
Simply apply the first law of Lemma~\ref{lem:x2x} repeatedly to obtain 
$m_1+ \dots + m_k\approx m_1+ \dots + m_k+2km_1m_2\dots m_k$, then apply $2\approx 4$ to obtain $m_1+ \dots + m_k\approx m_1+ \dots + m_k+2m_1m_2\dots m_k$, and then $x^2\approx x$ to reduce powers and obtain $m_1+ \dots + m_k\approx m_1+ \dots + m_k+2m$.
\end{proof}

Let $Y$ be a subset of the variables of a polynomial $p$.  Let $p|_{Y=0}$ denote the result of assigning all variables in $Y$ the value $0$.
\begin{defn}
Let $p$ be a polynomial in variables $\{x_1,\dots,x_n\}$ and $X\subseteq \{x_1,\dots,x_n\}$.  Then $p$ is \emph{very $X$-positive} if $p|_{Y=0}$ is a positive polynomial whenever $Y\subseteq \{x_1,\dots,x_n\}\backslash X$. 
\end{defn}
As an example, strictly positive polynomials are very $X$-positive for every subset $X$ of their variables, but $(x-1)^2+y$ is a positive polynomial (as the value of~$y$ is a lower bound), but is not very $\{x\}$-positive: when $y\mapsto 0$, the resulting polynomial $(x-1)^2$ is not positive at $x=1$.
\begin{lem}
Let $p$ be a positive polynomial factor of a strictly positive polynomial~$s$ and let $m$ be a monomial appearing in $s$ with variables from $X$.  
Then $p$ is very $X$-positive.
\end{lem}
\begin{proof}
On assignment of $0$ to variables not in $X$, the strictly positive polynomial $s$ has not reduced to $0$ due to the term $m$, so is still strictly positive.  
So no factor of $s$ can have a zero on the positive domain.
\end{proof}
Throughout, let $s$ be a positive polynomial, factorising into irreducible positive polynomials in $\mathbb{Z}[x_1,\dots,x_n]$ as $s= \prod(p_i - q_i)$ (where each $p_i,q_i$ are strictly positive and share no common monomial).  Let $s'$ be the strictly positive polynomial  $\prod(p_i + q_i)$.
This product is intended to have been expanded and simplified over $\mathbb{Z}[x_1,\dots,x_n]$ so that $s'$ is written as a sum of monomials.  We fix the notation $s'$ throughout the remaining lemmas.
\begin{lem}\label{lem:ss'}
$s$ and $s'$ are equivalent in the polynomial ring $\mathbb{Z}_2[x_1,\dots,x_n]$.
\end{lem}
\begin{proof}
This is because $p_i - q_i\equiv_2 p_i + q_i$ for each $i$.
\end{proof}

We say that a monomial \emph{appears} in a polynomial (written as a sum of monomials with no further cancellations possible) if it has nonzero coefficient.

\begin{lem}\label{lem:equivmod2}
Every monomial appearing in $s$ with coefficient $c\geq 1$ appears in $s'$ with coefficient $c'\geq c$ and $c\equiv_2 c'$.
\end{lem}
\begin{proof}
By Lemma~\ref{lem:ss'}, the coefficient $c'$ is equivalent to $c$ modulo $2$.  As the product $\prod(p_i + q_i)$ differs from $\prod(p_i - q_i)$ by replacing all subtractions by addition, the coefficients in $s'$ can only be greater or equal to those in $s$. 
\end{proof}
A monomial $m$ is \emph{strictly covered} by some monomials $m_1,\dots,m_n$ if it is covered by the monomials, but each individual monomial $m_i$ has its variables as a proper subset of $m$.
\begin{lem}
If a monomial $m$ appearing in $s'$ is strictly covered by some monomials in $s$, then $\HSI\cup\{x^2\approx x\}\vdash s\approx s+2m$.
\end{lem}
\begin{proof}
This follows immediately by Lemma \ref{lem:cover2} because strict covering of $m$ by some monomials $m_1,\dots,m_k$ requires $k>1$.
\end{proof}

\begin{lem}\label{lem:canincreasemod2}
If a monomial $m$ appearing in $s'$ is not strictly covered by some monomials in $s$, then either it has the same coefficient in $s$ as in $s'$ or  $\HSI\cup\{x^2\approx x\}\vdash s\approx s+2m$.
\end{lem}
\begin{proof}
Let $X=\{x_1,\dots,x_u\}$ denote the variables in $m$ and $Y=\{y_1,\dots,y_v\}$ the variables in $s$ that are not in $m$.
First we note that $m$ must be at least covered by monomials in $s$.  
To show this, let $Y$ denote the variables not appearing in $m$, and consider the law $s|_{Y=0}\approx s'|_{Y=0}$.  
Now $m$ appears in $s'|_{Y=0}$ and the same variables appear in  $s|_{Y=0}$. 
It follows that all of the variables in $m$ appear in monomials of $s$ whose content is a subset of that in $m$.
Equivalently, $m$ is covered by the monomials in $s|_{Y=0}$ and hence in $s$.

As $m$ is covered but not strictly covered, it follows that there is at least one monomial appearing in $s$ with the same variables as $m$.  
If a monomial with the same variables as $m$ (including possibly $m$ itself) appears in~$s$ with coefficient at least $2$, then the result follows using $2\approx 4$ and Lemma \ref{lem:equivmod2}.  
Similarly if there is more than one monomial $m_1,m_2$ in $s$ with the same variables as $m$, then $\HSI\cup\{x^2\approx x\}\vdash m_1\approx m_2\approx m$ and then $2\approx 4$ can again be used to obtain $s\approx s+2m$.   
Now assume the remaining case: there is a unique monomial $m_s$ in $s$ with the same variables as $m$, and that the coefficient of $m_s$ in $s$ is $1$.   
If there is a monomial $m'$ in $s$ whose variables are a proper subset of those in $m$ and $m_s$, then the law $x+xy\approx x+3xy$ from Lemma~\ref{lem:x2x} gives $m'+m_s\approx m'+3m_s$ and we are in a case already considered.  Note that it is allowed for $m'$ to be a constant (with set of variables $\varnothing$).  
So we can now assume that $m_s$ is the only monomial in $s$ whose variables are a subset of $X$ and the assumptions on $m_s$ in $s$ then imply that $s|_{Y=0}$ is simply $m_s$ by itself.  Our goal is to show that the coefficient of $m$ in $s'$ is also $1$.    
The factorisation of $s$ as $\prod_{1\leq i\leq k}(p_i-q_i)$ continues to hold after assigning variables in $Y$ to $0$, so each $p_i-q_i$ has $(p_i-q_i)|_{Y=0}$ equal to a factor of $m_s$: either a constant, or a monomial whose variables are within $X$.  
We denote the factor $(p_i-q_i)|_{Y=0}$ as $m_i$, and observe that the reduction of $p_i-q_i$ to $(p_i-q_i)|_{Y=0}$ is simply by replacing monomials with a variable in $Y$ by $0$, there are no proper subtractions that arise between nonzero coefficient monomials $m'-m'$, as these did not originally occur in $p_i-q_i$.
This in turn implies that the coefficient of $m_i$ cannot be negative in $(p_i-q_i)|_{Y=0}$, for otherwise in $p_i-q_i$, all monomials in $p_i$ involved at least one variable in $Y$: taking these values very close to $0$ with values for $X$ variables fixed then makes $p_i-q_i$ negative (yet it is a positive polynomial).
It follows that $m_i$ is the only monomial in $p_i$ whose variables are within $X$ (including the possibility that $m_i$ is a constant), and that all monomials in $q_i$ involve a variable from $Y$. 
Thus the unique monomial in either of $s=\prod_{1\leq i\leq k}(p_i-q_i)$ (which we denoted $m_s$) or $s'=\prod_{1\leq i\leq k}(p_i+q_i)$ (which must be $m$) with variables amongst $X$ is the monomial $m_1m_2\dots m_k$, and the coefficient is fixed identically in both cases.
\end{proof}

\begin{proof}[Proof of Theorem~\ref{thm:x2x}]
Let ${\bf M}$ be a model of $\HSI\cup\{x^2\approx x\}$.  We show that ${\bf M}$ can be expanded to a model of Wilkie's $\HSI^*$ by defining the operation $t_{p-q}$ for every positive polynomial $p-q$ (where $p$ and $q$ are strictly positive with no common monomials).  
Define $t_{p-q}(x_1,\dots,x_n)$ to be $p+q$.   As before, when $q$ is empty, the law $t_{p-q}\approx p$ holds directly, as $p+q$ is just $p$.  
When~$s$ is strictly positive, and factorises as $\prod_{1\leq i\leq k}(p_i-q_i)$, then we need to verify satisfaction of $t_s\approx \prod_{1\leq i\leq k}t_{p_i-q_i}$.  But this is exactly what is shown in Lemmas~\ref{lem:equivmod2} and Lemma~\ref{lem:canincreasemod2}, where $s'$ denotes $\prod_{1\leq i\leq k}t_{p_i-q_i}$ written as $\prod_{1\leq i\leq k}(p_i+q_i)$.  
Indeed, by Lemma~\ref{lem:equivmod2} the monomials appearing in $s'$ have the same coefficient as those appearing in $s$ modulo $2$, with those in $s'$ having greater or equal value.  There are no negative coefficients of monomials in $s$. If the coefficient of a monomial in $s$ is $0$, then Lemma~\ref{lem:cover2} applies.  If the coefficient is~$1$, then Lemma~\ref{lem:canincreasemod2} applies.  If the coefficient is $2$ or more, then $2\approx 4$ applies.
So $t_s\approx \prod_{1\leq i\leq k}t_{p_i-q_i}$ is a consequence of $\overline{\HSI}\cup\{x^2\approx x\}$ and the definitions of $t_s$ and~$t_{p_i-q_i}$.
Then by Wilkie's Theorem~\ref{thm:WilkiesTheorem}, all laws in $\Th_{\rm eq}(\langle\mathbb{N},+,\cdot,\uparrow,1\rangle)$ are consequences of $\HSI\cup\{x^2\approx x\}$.
\end{proof}

\begin{pro}\label{pro:3element}
All HSI algebras on at most $3$-elements lie within the variety of $\langle\mathbb{N},+,\cdot,\uparrow,1\rangle$.
\end{pro}
\begin{proof}
This is trivial for the integer models, as they are quotients of~$\langle\mathbb{N},+,\cdot,\uparrow,1\rangle$, while 41 of the 42 non-integer algebras in~\cite[Appendix~A]{burlee92} satisfy $x^2\approx x$ so are covered by Theorem~\ref{thm:x2x}.  The single model that fails $x^2\approx x$ is number 35, and we denote it by $M_{35}$:
\[
\begin{tabular}{c|ccc}
$+$&$1$&$2$&$b$\\
\hline
$1$&$2$&$2$&$2$\\
$2$&$2$&$2$&$2$\\
$b$&$2$&$2$&$2$\\
\end{tabular}\qquad 
\begin{tabular}{c|ccc}
$\cdot$&$1$&$2$&$b$\\
\hline
$1$&$1$&$2$&$b$\\
$2$&$2$&$2$&$2$\\
$b$&$b$&$2$&$2$\\
\end{tabular}\qquad 
\begin{tabular}{c|ccc}
$\uparrow$&$1$&$2$&$b$\\
\hline
$1$&$1$&$1$&$1$\\
$2$&$2$&$2$&$2$\\
$b$&$b$&$2$&$2$\\
\end{tabular} 
\]
We employ the approach used in \cite[Proposition~3.3]{alsjac2}.
Recall that the $1$-generated free algebra in the variety of ${\bf N}=\langle\mathbb{N},+,\cdot,\uparrow,1\rangle$, is the algebra of one-variable $\mathscr{L}$-functions over $\mathbb{N}$; we will denote it by $\mathbf{N}[x]$.  
We show that $M_{35}$ is a quotient of $\mathbf{N}[x]$.
 We define a 3-block partition of $\mathbf{N}[x]$, show that it is a congruence, and that the quotient is isomorphic to $M_{35}$. 
 The blocks of the partition are $B_1:=\{1\}$, $B_b:=\{ x\}$, $B_2:=\mathbf{N}[x]\backslash\{1,x\}$, where for $a\in \{1,2,b\}$, the  block $B_a$ corresponds to $a$, in the sense that the partition is a congruence, and the quotient structure on $\{B_1,B_2,B_b\}$ is isomorphic to $M_{35}$ under $B_a\mapsto a$.  
 To verify the congruence property for $+$, observe that in $\mathbf{N}[x]$, every proper sum is in $\mathbf{N}[x]\backslash\{1,x\}$, that is $B_2$.  This is precisely the addition table for $M_{35}$: all sums equal~$2$. 
For the congruence property for $\cdot$, observe that in $\mathbf{N}[x]$, all products except for those of the form $1\cdot t$ and $t\cdot 1$ (for $t\in \mathbf{N}[x]$) lie in $B_2$ again, with $1\cdot 1=1$, and $1\cdot x=x=x\cdot 1$.  
This is precisely the multiplication table for $M_{35}$.  
Finally for exponentiation, again all proper exponentiations lie in $B_2$, except for $1^t=1$ and $t^1$ (for any $t\in \mathbf{N}[x]$).  
For exponentiation of the form $t^1$,  only $1^1$ and $x^1$ gives values outside of $B_2$, giving $B_1$ and $B_b$ respectively.  
This is the exponentiation table for~$M_{35}$.  
So~$M_{35}$ is isomorphic to the quotient of  $\mathbf{N}[x]$ by the congruence with blocks $B_1,B_2,B_b$.
This also concludes the proof that $3$-element HSI algebras lie in the variety of $\langle\mathbb{N},+,\cdot,\uparrow,1\rangle$.
\end{proof}

Theorem~\ref{thm:x2x} also covers the significant majority of the 4 and 5-element models as well.  Mace4 quickly finds all models and multiplicatively idempotent models on 4 and 5 elements (up to isomorphism): 597 of 657 four-element models and 11,158 of 13,577 five-element models are found to be multiplicatively idempotent.

As observed by Burris and Lee~\cite[Proposition~5.3]{burlee92}, the variety of Brouwerian lattices is term equivalent to the subvariety of that defined by $\HSI$: the term equivalence is simply that $\leftarrow$ plays the role of exponentiation.  As these are multiplicatively idempotent, they fall within Theorem~\ref{thm:x2x}. 
\begin{cor}\label{cor:Brouwerian}
The variety of ${\bf N}$ contains the variety of Brouwerian lattices in signature $\{\vee,\wedge,\leftarrow,1\}$, and hence has continuum many subvarieties.
\end{cor}
\begin{proof}
Only the cardinality of the subvariety lattice is not immediate, however Wronski~\cite{wro} has shown that there are continuum many distinct varieties of Brouwerian lattices, and hence also of subvarieties of the variety of ${\bf N}$.
\end{proof}

A further natural model is the graph homomorphism classes.
Recall that for graphs $G=(V_1,E_1)$ and $H=(V_2,E_2)$, a homomorphism $\phi:G\to H$ is a map from the vertices $V_1$ to the vertices $V_2$ that preserves the edge relation: $\{u,v\}\in E_1$ implies $\{\phi(u),\phi(v)\}\in E_2$.
The graph homomorphism lattice is obtained by taking the classes of homomorphism equivalent graphs as elements, letting $\vee$ be the disjoint union construction, and $\wedge$ the graph direct product (the category product in the category of graphs with homomorphisms).  The single vertex looped graph acts as an identity element with respect to $\wedge$.
This lattice becomes a Brouwerian lattice, using graph exponentiation $G^H$, whose vertices are the functions from $V_2$ to $V_1$, with $(f,g)$ an edge if $\{u,v\}\in E_2$ implies $\{f(u),g(v)\}\in E_1$; see Hell and Ne\v{s}etr\v{\i}l~\cite[\S2.5]{helnes} for example.
As a Brouwerian lattice, Corollary~\ref{cor:Brouwerian} shows this well-studied algebraic object lies in the variety of ${\bf N}$.    
It seems quite remarkable that such a complicated structure can be obtained, up to isomorphism, as a quotient of a subalgebra of a power of ${\bf N}$.

\section{The large exponential algebra}\label{sec:LEA}
The $\{+,\cdot,1\}$-terms are precisely the polynomials.  
By a \emph{strict exponential term} we mean a term of the form $u^v$ where $u$ is not $1$ and $v$ is not constant.
An \emph{exponential term} will mean a product of terms, at least one of which is a strict exponential.
Every term may be reduced modulo $\HSI$ to a sum of a polynomial with a sum of exponential terms, simply by applying the distributivity laws left to right, then gathering polynomial summands together.  
For such a term $t$, we let $t_{\rm poly}$ denote the polynomial summand and $t_{\rm exp}$ the sum of exponential terms, so that $t\approx t_{\rm poly}+t_{\rm exp}$ is a valid identity of ${\bf N}$.
We mention that terms that define identical functions on~$\mathbb{N}$ may reduce in different ways, but the following lemma shows that any difference can only occur within the exponential summands, not the polynomial part.
\begin{lem}\label{lem:polysum}
If $s$ and $t$ are terms such that ${\bf N}\models s\approx t$, then ${\bf N}\models s_{\rm poly}\approx t_{\rm poly}$ (or both polynomials are empty) and ${\bf N}\models s_{\rm exp}\approx t_{\rm exp}$ (or both sums of exponentials are empty).
\end{lem}
\begin{proof}
This follows from Wilkie's normal form argument: in the reduction to polynomial in the expanded signature, the representation of $s_{\rm poly}$ is essentially left unchanged (so depends on no $x$-variables, only $y$-variables), while the reduction of $s_{\rm exp}$ depends always on at least some $x$-variables (in which $\tau$-exponentials are placed).  
The same is true of $t_{\rm poly}$ and $t_{\rm exp}$ and then ${\bf N}\models s\approx t$ and then Theorem~\ref{thm:samerep} shows that $s_{\rm poly}$ and $t_{\rm poly}$ are identical polynomials.  Cancellativity then implies that ${\bf N}\models s_{\rm exp}\approx t_{\rm exp}$ also.
\end{proof}
In this paper we only use Lemma~\ref{lem:polysum} when $s$ and $t$ are constant-free, and in this scenario there is a much simpler argument than Wilkie's~\cite{wil00} for Theorem~\ref{thm:samerep}, using only familiar properties of polynomials and prime numbers.  We give it, in case the methodology has other applications in small HSI algebras, and to make the work more self contained.  
The proof will make use of the fact that a nonzero polynomial $p(\vec{x})$ of total degree $d$ and in $k$ variables $\vec{x}=x_1,\dots,x_k$ cannot vanish on all of any subset $S^k$ of $\mathbb{Z}^k$ when $|S|>d$.  
This fact is a direct consequence of either the DeMillo-Lipton-Schwartz-Zippel (DMLSZ) Lemma or the Combinatorial Nullstellensatz, both of which have accessible proofs.
\begin{proof}[Alternative proof of Lemma~\ref{lem:polysum} in the constant-free case]
Let $k$ denote the number of variables.
The result is trivial if $s_{\rm exp}$ (or $t_{\rm exp}$) is empty, as a polynomial cannot be equal to an exponential, and if both are polynomial then there is nothing to prove.  
So we assume that $s_{\rm exp}$ and $t_{\rm exp}$ are not empty.  
If $s_{\rm poly}$ and $t_{\rm poly}$ are empty we are done, so assume that $s_{\rm poly}$ is not empty, and has been expanded to a sum of monomials.  
 Let $d_s$ denote the total degree of  $s_{\rm poly}$, and $n_s$ denote the sum of the coefficients.  
We may similarly choose $d_t$ and $n_t$ for the corresponding values in $t_{\rm poly}$, if it is nonempty (otherwise we take $t_{\rm poly}(x)$ to be $0$); let $d:=\max\{d_s,d_t\}$ and $n=\max\{n_s,n_t\}$.
Now select any prime $p$ greater than $\max\{n(d+1)^d,d\}$ and consider the set $S:=\{ip\mid i=1,\dots,d+1\}$.  
We show that $s_{\rm poly}$ and $t_{\rm poly}$ agree on $S^k$, which shows that they coincide (by either the DMLSZ Lemma or the Combinatorial Nullstellensatz).
For any $(m_1,\dots,m_k) \in S^k$, let $\vec{p}$ abbreviate $m_1p,\dots,m_kp$.
Observe that the choice of $p$ ensures that $s_{\rm poly}(\vec{p})$ and $t_{\rm poly}(\vec{p})$ (if it exists) are at most $n(d+1)p^d<p^p$.  
However $s_{\rm exp}(\vec{p})$  and $t_{\rm exp}(\vec{p})$ are multiples of $p^p$.  
It follows that $s(\vec{p})$ is not a multiple of $p^p$, and as $s(\vec{p})=t(\vec{p})$, it then follows that $t_{\rm poly}(\vec{p})$ is equal to $s_{\rm poly}(\vec{p})$, as required.
Hence $s_{\rm poly}$ is identically equal to $t_{\rm poly}$ everywhere.
Thus ${\bf N}^-\models s_{\rm poly}+s_{\rm exp}\approx t_{\rm poly}+t_{\rm exp}$ implies that ${\bf N}^-\models s_{\rm exp}\approx t_{\rm exp}$ as well.
\end{proof}
For two elements $s=s_{\rm poly}+s_{\rm exp}$ and $t=t_{\rm poly}+t_{\rm exp}$ we have (using only distributivity of $\cdot$ over $+$):
\begin{itemize}
	\item $(s+t)_{\rm poly} = s_{\rm poly}+t_{\rm poly}$ and $(s+t)_{\rm exp}=s_{\rm exp}+t_{\rm exp}$
	\item $(s\cdot t)_{\rm poly} = s_{\rm poly}\cdot t_{\rm poly}$ and $(s\cdot t)_{\rm exp}=s_{\rm poly}t_{\rm exp}+s_{\rm exp}t_{\rm poly}+s_{\rm exp}t_{\rm exp}$
	\item $(s^t)_{\rm poly}$ is empty but $(s^t)_{\rm exp}=s^t$.
\end{itemize}

Let ${\bf N}_{\rm poly}[X]$ denote the polynomial semiring over $\mathbb{N}$ in variables $X$: it is just the free $\overline{\HSI}$-algebra, freely generated by $X$.
Now let us consider a further symbol~$\infty$ and consider all formal expressions of the form $s$ and $s+\infty$ where $s\in {\bf N}_{\rm poly}[X]$, where $\infty$ alone corresponds to a degenerate case with empty polynomial part. 
These formal sums become an HSI algebra in a simple way using the rules $x\cdot \infty=\infty$ (from which we deduce that $\infty+\infty=2\infty=\infty$) and define exponentiation by letting $x^\infty=\infty$ for all $x$ not equal to $1$, $\infty^x=\infty$ for all $x$, and $s^t=\infty$ if $s$ is not $1$ and $t$ is not constant.
We call this the \emph{large exponential algebra} and denote it by ${\bf E}(X)$.
\begin{pro}
The large exponential algebra lies in the variety of ${\bf N}$.
\end{pro}
\begin{proof}
Let $F(X)$ be the free algebra for the variety of ${\bf N}$ with free generators the variables in the set $X$.
Every element may be represented as a $\mathscr{L}$-term which we may choose to be of the form $s_{\rm poly}+s_{\rm exp}$.
Lemma \ref{lem:polysum} shows that we may identify all exponential terms in a single block, $\infty$, and the corresponding equivalence relation is a congruence; the quotient is ${\bf E}(X)$.
\end{proof}

The large exponential algebra ${\bf E}(X)$ has many small quotients that are often easy to identify. The algebra $M_{35}$ in the proof of Proposition~\ref{pro:3element} can be seen as a quotient for example.
Any congruence $\theta$ of ${\bf N}_{\rm poly}[X]$ in which $\{1\}$ is a congruence class extends to a congruence on ${\bf E}(X)$ that agrees with $\theta$ in the polynomial part: $s+\infty$ is congruent to $t+\infty$ if $s\mathrel{\theta}t$.  
Because ${\bf N}_{\rm poly}[X]$ is just the free semiring, almost any commutative semiring extends to an HSI algebra by setting $s^t=\infty$ whenever $t$ is not a sum of $1$s and $s$ is not $1$.
We use ${\bf E}_-(X)$ to denote the constant free variant of ${\bf E}(X)$.

\section{RSI algebras: the constant-free signature}\label{sec:Martin}

Recall that ${\bf N}^{-}$ denotes the reduct $\langle
\mathbb{N},+,\cdot,\uparrow\rangle$ of ${\bf N}$ and let $\RSI$
denote those identities of ${\rm HSI}$ that do not involve the nullary
symbol $1$: the \emph{restricted (high) school identities}.  
We also let $\mathscr{L}^-$ denote the signature $\{+,\cdot,\uparrow\}$.
Martin~\cite{mar} showed that the identity
$(x^x+x^x)^y(y^y+y^y)^x\approx (x^y+x^y)^x(y^x+y^x)^y$ 
is satisfied by ${\bf N}^{-}$ but cannot be derived from $\RSI$, and used a generalisation of this law to show that there can be no finite axiomatisation for the equational theory of ${\bf N}^-$. 
We mention that Martin's proof that this law and generalisations are not derivable and the nonfinite axiomatisability argument are quite involved~\cite[pp.~99--118]{mar}, involving a careful analysis of the eventual dominance relation on one variable terms, and ultimately resting on the transcendality of the real number $e$.
In this section we present a counterexample small enough to be both human checkable and human findable (we do not address the nonfinite axiomatisability though).
Throughout this section, by
``counterexample'', we mean a model of $\RSI$ on which some valid
identity of ${\bf N}^{-}$ fails: in other words, a constant-free Gurevi\u{c} algebra.

The following identity implies the first mentioned example by Martin and is the actual base case of his infinite set used in the nonfinite axiomatisability result (page~97 of \cite{mar}; we have reordered the variables to appear in alphabetical order): 
\[
M(w,x,y,z):=
(w^x+w^x)^y(z^y+z^y)^x\approx (w^y+w^y)^x(z^x+z^x)^y.
\]  
This
identity is also satisfied by ${\bf N}^{-}$ as ${\bf N}^{-}$ is a reduct of
${\bf N}$, and using ${\rm HSI}$, both sides reduce to
$2^{x+y}w^{xy}z^{xy}$.  

The
similarities between $M(w,x,y,z)$ and Wilkie's identity $W(x,y)$ are
obvious, however the extra symmetry and simplicity of $M(w,x,y,z)$ over
$W(x,y)$ as well as the reduced nature of $\RSI$ compared to $\HSI$ enables considerably smaller counterexamples.  
The second author found the following 6-element RSI algebra by hand in the 1990s, which is a counterexample to $M(w,x,y,z)$ (the
identity fails at the $4$-tuple $(a,b,c,d)$):\\

\noindent \begin{tabular}{c|cccccc}
$+$&$a$&$b$&$c$&$d$&$e$&$f$\\
\hline
$a$&$f$&$f$&$f$&$f$&$f$&$f$\\
$b$&$f$&$d$&$f$&$f$&$f$&$f$\\
$c$&$f$&$f$&$a$&$f$&$f$&$f$\\
$d$&$f$&$f$&$f$&$f$&$f$&$f$\\
$e$&$f$&$f$&$f$&$f$&$f$&$f$\\
$f$&$f$&$f$&$f$&$f$&$f$&$f$\\
\end{tabular}\ \
\begin{tabular}{c|cccccc}
$\cdot$&$a$&$b$&$c$&$d$&$e$&$f$\\
\hline
$a$&$f$&$f$&$f$&$f$&$f$&$f$\\
$b$&$f$&$f$&$e$&$f$&$f$&$f$\\
$c$&$f$&$e$&$f$&$f$&$f$&$f$\\
$d$&$f$&$f$&$f$&$f$&$f$&$f$\\
$e$&$f$&$f$&$f$&$f$&$f$&$f$\\
$f$&$f$&$f$&$f$&$f$&$f$&$f$\\
\end{tabular}\ \
\begin{tabular}{c|cccccc}
$\uparrow$&$a$&$b$&$c$&$d$&$e$&$f$\\
\hline
$a$&$f$&$ b $&$f$&$f$&$f$&$f$\\
$b$&$f$&$f$&$f$&$f$&$f$&$f$\\
$c$&$f$&$f$&$f$&$f$&$f$&$f$\\
$d$&$f$&$f$&$ c $&$f$&$f$&$f$\\
$e$&$f$&$f$&$f$&$f$&$f$&$f$\\
$f$&$f$&$f$&$f$&$f$&$f$&$f$\\
\end{tabular}\\

We denote this algebra by ${\bf M}$.
The simplicity of ${\bf M}$ enables quick and thorough
checking of the identities in $\RSI$ by hand.  
Commutativity of
$+$ and
$\cdot$ are trivial, while routine observations show that all
other terms appearing in $\RSI$ laws are constantly
equal to $f$.  
For example, for every $i,j,k\in M$, we have 
\[
(ij)^k\in
(M\cdot M)\uparrow M=\{e,f\}\uparrow M=\{f\},
\]
while if $i^kj^k\neq f$
then
$\{i^k,j^k\}=\{b,c\}$, but it is easily seen that no $k\in M$ exists with this
property (no column in the $\uparrow$ table has both a $b$ and a $c$). 
This domination by the element $f$ reflects the fact that commutativity
of $+$ and $\cdot$ are the only identities of
$\RSI$ that can be applied to a term in the operations $+,\cdot,
\uparrow$ that is identically equal on $\mathbf{N}^-$ to either side of
$M(w,x,y,z)$.  This fact can also be turned into an easy syntactic proof of
the non-derivability of
$M(w,x,y,z)$ from $\RSI$.

It can be seen that ${\bf M}$ is generated by the elements $b,c$, and thus there is a
two variable law of ${\bf N}^{-}$ that fails on ${\bf M}$.  As $a=c+c$
and
$d=b+b$, the simplest of the possibilities is probably
\begin{eqnarray*}
\lefteqn{\left( (x+x)^y+(x+x)^y\right)^x\left(
(y+y)^x+(y+y)^x\right)^y}&&\\
&\phantom{abcdefdghabcdefdgh}\approx 
\left( (x+x)^x+(x+x)^x\right)^y\left( (y+y)^y+(y+y)^y\right)^x.&
\end{eqnarray*} 

One can easily give a lower bound of $4$ for the number of elements in
any counterexample to $(w^x+w^x)^y(z^y+z^y)^x\approx
(z^x+z^x)^y(w^y+w^y)^x$.  In the following lemmata, we assume that
${\bf A}$ is a model of $\RSI$ on which $M(w,x,y,z)$ fails at
$(a,b,c,d)$.  The first lemma is trivial.
\begin{lem}\label{lem:var}
$a\neq d$ and $b\neq c$.
\end{lem}
\begin{lem}\label{lem:mult}
$a,d$ are multiplicatively prime.
\end{lem}
\begin{proof}
Say $a=ef$ for some $e,f\in A$.  Then using $\RSI$ we get
$(a^b+a^b)^c(d^c+d^c)^b=
(e^bf^b+e^bf^b)^c(d^c+d^c)^b
=e^{bc}(f^b+f^b)^c(d^c+d^c)^b
=(f^b+f^b)^c(d^ce^c+d^ce^c)^b
=(f^b+f^b)^c(e^c+e^c)^bd^{bc}
=(f^bd^b+f^bd^b)^c(e^c+e^c)^b
=(d^b+d^b)^c(e^c+e^c)^bf^{bc}
=(d^b+d^b)^c(e^cf^c+e^cf^c)^b
=(d^b+d^b)^c(a^c+a^c)^b$.  The case when $d$ is multiplicatively
composite follows by symmetry.
\end{proof}
\begin{lem}\label{lem:add}
$b,c$ are additively prime.
\end{lem}
\begin{proof}
Say $b=e+f$ for some $e,f\in A$.  Then
$(a^b+a^b)^c(d^c+d^c)^b
=(a^ea^f+a^ea^f)^c(d^c+d^c)^e(d^c+d^c)^f
=a^{ec}(a^f+a^f)^c(d^c+d^c)^e(d^c+d^c)^f
=(a^f+a^f)^c(d^ca^c+d^ca^c)^e(d^c+d^c)^f
=(a^f+a^f)^c(a^c+a^c)^ed^{ce}(d^c+d^c)^f
=(a^fd^e+a^fd^e)^c(a^c+a^c)^e(d^c+d^c)^f
=a^{fc}(d^e+d^e)^c(a^c+a^c)^e(d^c+d^c)^f
=(d^e+d^e)^c(a^c+a^c)^e(d^ca^c+d^ca^c)^f
=(d^e+d^e)^c(a^c+a^c)^e(a^c+a^c)^fd^{cf}
=(d^ed^f+d^ed^f)^c(a^c+a^c)^e(a^c+a^c)^f
=(d^b+d^b)^c(a^c+a^c)^b$.  The case when $c$ is additively composite
follows by symmetry.
\end{proof}
\begin{lem}
${\bf A}$ has at least $4$ elements.
\end{lem}
\begin{proof}
As $M(w,x,y,z)$ fails at $(a,b,c,d)$, we have
$(a^b+a^b)^c(d^c+d^c)^b\neq  (a^c+a^c)^b(d^b+d^b)^c$ and at least
one of $a^b+a^b\neq d^b+d^b$ and $d^c+d^c\neq a^c+a^c$ holds.  Thus
there are at least two multiplicatively composite elements and two
additively composite elements.  Combining this with Lemmata
\ref{lem:var}, \ref{lem:mult} and \ref{lem:add}, it follows that there are
at least four elements in $A$.
\end{proof}
A ``by-hand'' proof that pushes this lower bound beyond 4 appears viable, but it seems unobvious how to avoid a lengthy and inelegant case analysis.
Instead we turned to Mace4, which finds no examples on fewer than 6 elements, but after just under 6000 seconds (on a Macbook Air M3, starting from size 2 and working upwards\footnote{After a further 5 million seconds, Mace4 was showing 28 examples on 6 elements, though isomorphism testing was not performed as the search had not finished.}) returns the following counterexample.

\noindent \begin{tabular}{c|cccccc}
$+$&$a$&$b$&$c$&$d$&$e$&$f$\\
\hline
$a$&$c$&$e$&$c$&$c$&$f$&$f$\\
$b$&$e$&$f$&$f$&$f$&$f$&$f$\\
$c$&$c$&$f$&$c$&$c$&$f$&$f$\\
$d$&$c$&$f$&$c$&$c$&$f$&$f$\\
$e$&$f$&$f$&$f$&$f$&$f$&$f$\\
$f$&$f$&$f$&$f$&$f$&$f$&$f$\\
\end{tabular}\ \
\begin{tabular}{c|cccccc}
$\cdot$&$a$&$b$&$c$&$d$&$e$&$f$\\
\hline
$a$&$c$&$d$&$c$&$c$&$c$&$c$\\
$b$&$d$&$c$&$c$&$c$&$c$&$c$\\
$c$&$c$&$c$&$c$&$c$&$c$&$c$\\
$d$&$c$&$c$&$c$&$c$&$c$&$c$\\
$e$&$c$&$c$&$c$&$c$&$c$&$c$\\
$f$&$c$&$c$&$c$&$c$&$c$&$c$\\
\end{tabular}\ \
\begin{tabular}{c|cccccc}
$\uparrow$&$a$&$b$&$c$&$d$&$e$&$f$\\
\hline
$a$&$c$&$e$&$c$&$c$&$c$&$c$\\
$b$&$e$&$c$&$c$&$c$&$c$&$c$\\
$c$&$c$&$c$&$c$&$c$&$c$&$c$\\
$d$&$c$&$c$&$c$&$c$&$c$&$c$\\
$e$&$c$&$c$&$c$&$c$&$c$&$c$\\
$f$&$a$&$b$&$c$&$e$&$d$&$c$\\
\end{tabular}\\
It is routine to verify that the left hand side of $M(a,a,b,b)$ takes the value $c$, while the right hand side takes the value $e$.  Note that this also implies that the example fails the first mentioned version of Martin's  law $M(x,x,y,y)$.

Ehrenfeucht \cite{ehr} proved that the one variable $\mathscr{L}$-terms
are well ordered by the relation of eventual dominance (see also~\cite[Theorem~A.5]{alsjac2}).  Gurevi\u{c}
conjectured~\cite[Remark~2 on page~29]{gur90} that the smallest one variable identity of ${\bf
N}$ not derivable from
$\HSI$ is 
\begin{eqnarray*}
\lefteqn{((1+x)^x+(2+x)^x)^{2^x}((2+x+x^3)^{2^x}+
(4+x^2+x^3)^{2^x})^x}&&\\
&\phantom{abcdefdgh}\approx 
((1+x)^{2^x}+(2+x)^{2^x})^x((2+x+x^3)^x+(4+x^2+x^3)^x)^{2^x}&
\end{eqnarray*} 
For terms in the restricted language $\{+,\cdot,\uparrow\}$ we conjecture
that $M(x,x,x^x,x+x)$ is the smallest exotic identity.  Mace4 takes less than 500 seconds to find a 7-element counterexample of this law, with comparable simplicity to our counterexample for $M(w,x,y,z)$.

There are many other basic variants.  The law
\[
M'(w,x,y,z):= ((w+w)^x+(y+y)^x)^z (w^z+y^z)^x \approx (w^x+y^x)^z ((w+w)^z+(y+y)^z)^x
\]
is closer to the theme of Wilkie's original law (both sides are equal to $2^{xz}(w^x+y^x)^z(w^z+y^z)^x$), while a skew version of Martin's law is:
\[
M''(w,x,y,z)'':=(w^x+(w+w)^x)^y(z^y+(z+z)^y)^x\approx (w^y+(w+w)^y)^x(z^x+(z+z)^x)^y.
\]
Both $M'(w,x,y,z)$ and $M''(w,x,y,z)$ fail on the second displayed counterexample for Martin's identity at $w=x=a$ and $y=z=b$.  
Mace4 returns the same model, so no smaller counterexample exists.  The first displayed counterexample satisfies both these laws, but fails $M(w,x,y,z)$, so $M(w,x,y,z)$ does not follow from $\RSI\cup\{M'(w,x,y,z)\}$.
\section{Two element models of \textsf{RSI}}
\begin{figure}
\begingroup
\setlength{\arraycolsep}{5pt}
\renewcommand{\arraystretch}{1.65}
\[
\begin{array}{@{}*{6}{c}@{}}
\triple{1}{\Taaaa}{\Taaaa}{\Taaaa} &
\triple{2}{\Taaaa}{\Taaab}{\Taaaa} &
\triple{3}{\Taaaa}{\Taaab}{\Taabb} &
\triple{4}{\Taaaa}{\Taaab}{\Tbbbb} &
\triple{5}{\Taaab}{\Taaaa}{\Taaaa} &
\triple{6}{\Taaab}{\Taaab}{\Taaaa}
\\\\
\triple{7}{\Taaab}{\Taaab}{\Taaab} &
\triple{8}{\Taaab}{\Taaab}{\Taabb} &
\triple{9}{\Taaab}{\Taaab}{\Tbbbb} &
\triple{10}{\Taaab}{\Tabbb}{\Taaaa} &
\triple{11}{\Taaab}{\Tabbb}{\Taaba} &
\triple{12}{\Taaab}{\Tabbb}{\Taabb}
\\\\
\triple{13}{\Taaab}{\Tabbb}{\Tbbbb} &
\triple{14}{\Taaab}{\Tbbbb}{\Tbbbb} &
\triple{15}{\Tabba}{\Taaaa}{\Taaaa} &
\triple{16}{\Tabba}{\Taaab}{\Taaaa} &
\triple{17}{\Tabba}{\Taaab}{\Taabb} &
\triple{18}{\Tabba}{\Taaab}{\Tbbbb}
\end{array}
\]
\endgroup\caption{The $2$-element models of \RSI}\label{fig:2element}
\end{figure}


Mace4 finds 18 distinct 2-element models of \RSI; see Figure \ref{fig:2element}, where they are displayed in lexicographic order.  
(It also finds there are 445 models on $3$ elements and 16,535 on $4$ elements.)
Note that, to save space, we list the tables for addition, then multiplication, then exponentiation and have suppressed the row and column labels.  Thus $M_1$ is really the following algebra:
\[
\begin{tabular}{c|cc}
$+$&$a$&$b$\\
\hline
$a$&$a$&$a$\\
$b$&$a$&$a$
\end{tabular}\qquad \begin{tabular}{c|cc}
$\cdot$&$a$&$b$\\
\hline
$a$&$a$&$a$\\
$b$&$a$&$a$
\end{tabular}\qquad\begin{tabular}{c|cc}
$\uparrow$&$a$&$b$\\
\hline
$a$&$a$&$a$\\
$b$&$a$&$a$
\end{tabular}
\]

\begin{pro}\label{pro:2RSI}
Every $2$-element RSI algebra lies in the variety of ${\bf N}^-$.
\end{pro}
\begin{proof} The proof covers the remainder of the section.
 Five of the 18 models are reducts of $2$-element models of \HSI, hence lie in the variety of $\langle \mathbb{N},+,\cdot,\uparrow\rangle$ by~\cite{asa} (or from Theorem~\ref{thm:x2x}: these are $M_3$, $M_8$, $M_{11}$, $M_{12}$, $M_{17}$).  
 We now give arguments to show that all of the remaining $13$ also lie in the variety of $\langle \mathbb{N},+,\cdot,\uparrow\rangle$, so are not counterexamples to any true law $\langle \mathbb{N},+,\cdot,\uparrow\rangle$.  
 Theorem~\ref{thm:x2x} cannot be applied directly as Wilkie's reduction (on which the proof depends) makes intrinsic use of the presence of constants.
Nevertheless, in all of the idempotent cases we are able to indirectly apply Theorem~\ref{thm:x2x} in the following way: if an RSI algebra ${\bf A}$ is a subreduct of a multiplicatively idempotent HSI algebra ${\bf B}$, then Theorem~\ref{thm:x2x} shows that ${\bf B}$ is in the variety of ${\bf N}$, and so in the reduct signature $\mathscr{L}^-$, the RSI algebra~${\bf A}$ is in the variety of ${\bf N}^-$.  
We also give direct proofs where they are available.  
 
 We let ${\bf N}^-[x]$ denote the free algebra in the variety of ${\bf N}^-$, noting that it consists of the $1$-variable $\mathscr{L}^-$-term functions on $\mathbb{N}$.
Each element can be represented as a term in $x$; while distinct terms can nontrivially represent the same term function (the two sides of Martin's identity for example), we will not encounter any challenging examples in the proofs. 
 The constant-free large exponential algebra ${\bf E}_-(x)$ will also be used frequently.

\subsection{$M_1$}  Let $S$ be the subalgebra of ${\bf N}^-$ consisting of all numbers greater than $1$.  Now take the quotient of $S$ determined by identifying all numbers $3$ and greater.  This is isomorphic to $M_1$ under the map that assigns $2$ to $b$ and the equivalence class $S\backslash\{2\}$ to $a$.

\subsection{$M_2$} Mace4 finds this as a subreduct of the following 4-element idempotent HSI algebra:
\[
\begin{tabular}{r|rrrr}
$+$ &  1 & 2 & $a$ & $b$ \\
\hline
    1 &  2 & 2 & $a$ & $a$ \\
    2 & 2 & 2 & $a$ & $a$ \\
    $a$ & $a$ & $a$ & $a$ & $a$ \\
    $b$ & $a$ & $a$ & $a$ & $a$ 
\end{tabular} \hspace{.5cm}
\begin{tabular}{r|rrrr}
$\cdot$ & 1 & 2 & $a$ & $b$ \\
\hline
    1  & 1 & 2 & $a$ & $b$\\
    2 & 2 & 2 & $a$ & $a$ \\
    $a$ & $a$ & $a$ & $a$ & $a$ \\
    $b$ & $b$ & $a$ & $a$&  $b$ 
\end{tabular} \hspace{.5cm}
\begin{tabular}{r|rrrr}
$\uparrow$ & 1 & 2 & $a$ & $b$ \\
\hline
    1 & 1 & 1 & 1 & 1 \\
    2 & 2 & 2 & $a$ & $a$ \\
    $a$ & $a$ & $a$ & $a$ & $a$ \\
    $b$ & $b$ & $b$ & $a$ & $a$ 
\end{tabular}
\]
There is an easy direct proof also.
Consider the subset $B:=\{x^i\mid i\geq 1\}$ of ${\bf N}^-[x]$.  It is routine to demonstrate that the equivalence relation with blocks equal to $B$ and $A:={\bf N}^-[x]\backslash B$ is a congruence; the quotient is isomorphic to $M_2$ under the map that assigns $B$ to $b$ and $A$ to $a$.  The following observations give the justification.
\begin{itemize}
\item No expression of the form $x^n$ arises as the sum of two elements of ${\bf N}^-[x]$.  
Certainly the sum of two additive polynomials cannot be equivalent to one involving no additions; all other terms have faster than polynomial growth, so cannot add to any positive function to arrive at $x^n$.
\item A product $u\cdot v$ in ${\bf N}^-[x]$ equals $x^n$ if and only if $u=x^{i}$ and $v=x^j$ and $i+j=n$ (so that $B\cdot B\subseteq B$ but $B\cdot A=A\cdot B\subseteq A$ and $A\cdot A\subseteq A$). 
\item Exponentiation is straightforward as the elements of $B$ do not have exponential growth. 
\end{itemize}

\subsection{$M_3$} This is an HSI algebra reduct, isomorphic to algebra $4$ in Burris and Lee~\cite[Theorem~4.1]{burlee92} (a quotient of ${\bf N}$).

\subsection{$M_4$} Mace4 finds this as a subreduct of the following $4$-element multiplicatively idempotent HSI algebra:
\[
\begin{tabular}{r|rrrr}
$+$ & $1$ & $2$ & $a$ & $b$\\
\hline
    $1$ & $2$ & $2$ & $2$ & $2$ \\
    $2$ & $2$ & $2$ & $2$ & $2$ \\
    $a$ & $2$ & $2$ & $a$ & $a$ \\
    $b$ & $2$ & $2$ & $a$ & $a$
\end{tabular} \hspace{.5cm}
\begin{tabular}{r|rrrr}
$\cdot$ & $1$ & $2$ & $a$ & $b$\\
\hline
    $1$ & $1$ & $2$ & $a$ & $b$ \\
    $2$ & $2$ & $2$ & $a$ & $a$ \\
    $a$ & $2$ & $2$ & $a$ & $a$ \\
    $b$ & $1$ & $2$ & $a$ & $b$
\end{tabular} \hspace{.5cm}
\begin{tabular}{r|rrrr}
$\uparrow$ & $1$ & $2$ & $a$ & $b$\\
\hline
    $1$ & $1$ & $1$ & $1$ & $1$ \\
    $2$ & $2$ & $2$ & $1$ & $1$ \\
    $a$ & $a$ & $a$ & $b$ & $b$ \\
    $b$ & $b$ & $b$ & $b$ & $b$
\end{tabular}
\]
We do not give a direct proof, but make an observation that $M_4$ records a variant to the property in Lemma~\ref{lem:polysum}, albeit in the constant-free one-variable case.
Let $A$ consist of the set of all terms in a single variable $x$ that after left to right applications of distributivity are a proper sum: a sum of at least two subterms.  
Let $B$ denote all other terms, which after the same reduction are not a sum of terms.  
Note that $B$ contains the term $x$, so generates all terms in this one variable by application of the operations in $\mathscr{L}^-$. 
Evaluating $x$ at $b$ in $M_2$, one can see that terms in $A$ evaluate to $a$ and those in $B$ evaluate to $b$.
The fact that algebra $M_4$ is in the variety of ${\bf N}^-$ can then be seen equivalent to the fact that there is no valid law between a term in $A$ and a term in $B$. 
This can be thought of as a (1-variable, constant-free) variation of Lemma~\ref{lem:polysum}, now separating the nontrivial sum terms from the trivial sum terms.

\subsection{$M_5$} This is a little like an additive version of $M_2$.  
This time we let $B$ denote those functions of the form $nx$ for some $n\in\mathbb{N}$ (which abbreviates $\overbrace{x+x+\dots+x}^{n\text{ times}}$) and $A:={\bf N}^-[x]\backslash B$ as before.  
Again $b$ corresponds to $B$ and $a$ to $A$.  
The fact that this partition is a congruence follows because the sum of any two strictly linear functions is strictly linear, (matching $b+b=b$), but all other sums, products and exponentials are not.
This algebra can also be seen as a quotient of the constant-free large exponential algebra.

\subsection{$M_6$} Now let $B\subseteq {\bf N}^-[x]$ denote the polynomial members of ${\bf N}^-[x]$.  The remaining elements of ${\bf N}^-[x]$ have non-polynomially bounded growth; denote them by $A$.  
The corresponding quotient matches $M_6$ with $a\mapsto A$ and $b\mapsto B$.  
We mention that this is the quotient of the large exponential algebra ${\bf E}_-(x)$ obtained by collapsing all elements involving $\infty$ (to $a$), and all those not involving $\infty$ to $b$.
Mace4 also finds $M_6$ as a subreduct of the following $4$-element idempotent HSI algebra, with $M_6$ the constant-free part.
\[
\begin{tabular}{r|rrrr}
$+$ & $1$ & $2$ & $a$ & $b$\\
\hline
    $1$ & $2$ & $1$ & $a$ & $b$ \\
    $2$ & $1$ & $2$ & $a$ & $b$ \\
    $a$ & $a$ & $a$ & $a$ & $a$ \\
    $b$ & $b$ & $b$ & $a$ & $b$
\end{tabular} \hspace{.5cm}
\begin{tabular}{r|rrrr}
$\cdot$ & $1$ & $2$ & $a$ & $b$\\
\hline
    $1$ & $1$ & $2$ & $a$ & $b$ \\
    $2$ & $2$ & $2$ & $a$ & $b$ \\
    $a$ & $a$ & $a$ & $a$ & $a$ \\
    $b$ & $b$ & $b$ & $a$ & $b$
\end{tabular} \hspace{.5cm}
\begin{tabular}{r|rrrr}
$\uparrow$ & $1$ & $2$ & $a$ & $b$\\
\hline
    $1$ & $1$ & $1$ & $1$ & $1$ \\
    $2$ & $2$ & $2$ & $2$ & $2$ \\
    $a$ & $a$ & $a$ & $a$ & $a$ \\
    $b$ & $b$ & $b$ & $a$ & $a$
\end{tabular}\]

\subsection{$M_7$} This corresponds to the algebra $\langle \{0,1\},\wedge,\wedge,\wedge\rangle$.  An identity holds on this algebra if and only if it has the same variables on both sides.  
This is trivially seen to hold for functions on $\mathbb{N}$ built from variables and the operations in $\mathscr{L}^-$ because these operations are strictly increasing on inputs larger than $1$, and depend on all variables.  
Thus $M_7$ lies in the variety of~$\mathbb{N}^-$.
We may also find $M_7$ as a subreduct of the following $3$-element idempotent HSI algebra (where $b:=2$):
\[
\begin{tabular}{r|rrr}
$+$ & $1$ & $2$ & $a$\\
\hline
    $1$ & $2$ & $2$ & $a$ \\
    $2$ & $2$ & $2$ & $a$ \\
    $a$ & $a$ & $a$ & $a$
\end{tabular} \hspace{.5cm}
\begin{tabular}{r|rrr}
$\cdot$ & $1$ & $2$ & $a$\\
\hline
    $1$ & $1$ & $2$ & $a$ \\
    $2$ & $2$ & $2$ & $a$ \\
    $a$ & $a$ & $a$ & $a$
\end{tabular} \hspace{.5cm}
\begin{tabular}{r|rrr}
$\uparrow$ & $1$ & $2$ & $a$\\
\hline
    $1$ & $1$ & $1$ & $1$ \\
    $2$ & $2$ & $2$ & $a$ \\
    $a$ & $a$ & $a$ & $a$
\end{tabular}
\]

\subsection{$M_8$}  This is an $HSI$ algebra reduct, isomorphic to algebra $3$ in Burris and Lee~\cite[Theorem~4.1]{burlee92}.

\subsection{$M_9$}   Mace4 finds this as a subreduct of a $4$-element multiplicatively idempotent HSI algebra:
\[
\begin{tabular}{r|rrrr}
$+$ & 1 & $a$ & $b$ & $c$\\
\hline
    1 & 1 & $c$ & 1 & $c$ \\
    $a$ & $c$ & $a$ & $a$ & $c$ \\
    $b$ & 1 & $a$ & $b$ & $c$ \\
    $c$ & $c$ & $c$ & $c$ & $c$
\end{tabular} \hspace{.5cm}
\begin{tabular}{r|rrrr}
$\cdot$ & 1 & $a$ & $b$ & $c$\\
\hline
    1 & 1 & $a$ & $b$ & $c$ \\
    $a$ & $a$ & $a$ & $a$ & $a$ \\
    $b$ & $b$ & $a$ & $b$ & $a$ \\
    $c$ & $c$ & $a$ & $a$ & $c$
\end{tabular} \hspace{.5cm}
\begin{tabular}{r|rrrr}
$\uparrow$ & 1 & $a$ & $b$ & $c$\\
\hline
    1 & 1 & 1 & 1 & 1 \\
    $a$ & $a$ & $b$ & $b$ & $a$ \\
    $b$ & $b$ & $b$ & $b$ & $b$ \\
    $c$ & $c$ & 1 & 1 & $c$
\end{tabular}
\]
As with all algebras here (and as was detailed in the case of $M_4$), the presence of $M_{9}$ in the variety of ${\bf N}^-$ records a nontrivial fact about one-variable $\mathscr{L}^-$-functions on $\mathbb{N}$.  
The algebra is generated by $a$, and $\{a\}$ forms a $\{+,\cdot\}$-subalgebra.  
So as a homomorphic image of ${\bf N}^-[x]$ (with $x\mapsto a$), all polynomials map to $a$, but so also does any polynomial multiplied by any term (as $a\cdot a=a\cdot b=a$).  
All strict exponentials, products of strict exponentials and sums of products of strict exponentials map to $b$ (because all powers equal $b$ and $\{b\}$ is a $\{+,\cdot\}$-subalgebra).
It follows that the elements $a$ and $b$ separate out sums of products of strict exponentials ($b$) from those elements that have a polynomial summand or have a polynomial factor.  
The presence of $M_{9}$ in the variety of ${\bf N}^-$ is equivalent to the fact that there is no law between these two classes of terms.

\subsection{$M_{10}$}  Mace4 finds this as a subreduct of the following $3$-element multiplicatively idempotent HSI algebra:
\[
\begin{tabular}{r|rrr}
$+$ & $1$ & $a$ & $b$\\
\hline
    $1$ & $1$ & $1$ & $1$ \\
    $a$ & $1$ & $a$ & $a$ \\
    $b$ & $1$ & $a$ & $b$
\end{tabular} \hspace{.5cm}
\begin{tabular}{r|rrr}
$\cdot$ & $1$ & $a$ & $b$\\
\hline
    $1$ & $1$ & $a$ & $b$ \\
    $a$ & $1$ & $a$ & $b$ \\
    $b$ & $1$ & $b$ & $b$
\end{tabular} \hspace{.5cm}
\begin{tabular}{r|rrr}
$\uparrow$ & $1$ & $a$ & $b$\\
\hline
    $1$ & $1$ & $1$ & $1$ \\
    $a$ & $a$ & $a$ & $a$ \\
    $b$ & $b$ & $a$ & $a$
\end{tabular}
\]

\subsection{$M_{11}$}  This is an HSI algebra reduct, isomorphic to algebra $1$ in Burris and Lee~\cite[Theorem~4.1]{burlee92}.

\subsection{$M_{12}$}  This is an HSI algebra reduct, isomorphic to algebra $2$ in Burris and Lee~\cite[Theorem~4.1]{burlee92}. 

\subsection{$M_{13}$}  Mace4 finds this as a subreduct of the following $3$-element idempotent HSI algebra:
\[
\begin{tabular}{r|rrr}
$+$ & $1$ & $a$ & $b$\\
\hline
    $1$ & $1$ & $a$ & $a$ \\
    $a$ & $a$ & $a$ & $a$ \\
    $b$ & $a$ & $a$ & $b$
\end{tabular} \hspace{.5cm}
\begin{tabular}{r|rrr}
$\cdot$ & $1$ & $a$ & $b$\\
\hline
    $1$ & $1$ & $a$ & $b$ \\
    $a$ & $a$ & $a$ & $b$ \\
    $b$ & $b$ & $b$ & $b$
\end{tabular} \hspace{.5cm}
\begin{tabular}{r|rrr}
$\uparrow$ & $1$ & $a$ & $b$\\
\hline
    $1$ & $1$ & $1$ & $1$ \\
    $a$ & $a$ & $b$ & $b$ \\
    $b$ & $b$ & $b$ & $b$
\end{tabular}
\]
While this example is sufficient to complete case $M_{13}$, we additionally give two direct proofs that $M_{13}\in\mathsf{V}(\mathbb{N}^-)$, in order to  develop some techniques that can be adapted to cases where the idempotence approach is not available.  
First, observe that $M_{13}$ is isomorphic to the quotient of the constant free large exponential algebra~${\bf E}^-(x)$ by declaring two formal sums $s+\infty$ (or just $s$) and $t+\infty$ (or just $t$) equal if either $s$ and $t$ are both empty (element $b$), or if $s$ and $t$ are nonempty (element $a$).

Here is a further approach to proving $M_{13}\in\mathbb{V}({\bf N}^-)$ more succinctly captures the prime choice in the constant-free proof of Lemma~\ref{lem:polysum}.  
Let $\overline{\mathbb{N}}$ denote a nonstandard model of $\mathbb{N}$: some nontrivial ultrapower of $\mathbb{N}$.  
Then $\overline{\mathbb{N}}$ contains a copy of $\mathbb{N}$, but also infinite numbers: numbers that are greater than any $n\in \mathbb{N}$.  
Let $P$ be any infinite prime, and consider the $\mathscr{L}^-$-subalgebra ${\bf M}$ of $\langle\overline{\mathbb{N}};+,\cdot,\uparrow\rangle$ generated by~$P$.  
Now let $a$ denote those elements of ${\bf M}$ that are not multiples of $P^P$ and $b$ those that are multiples of $P^P$.
The additive, multiplicative and exponential behaviour of $a,b$ is exactly as in the table for $M_{13}$: note that we retain $\mathbb{N}$ as our ambient model of arithmetic: so finite sums of copies of $P$ is of the form $nP$ for $n\in \mathbb{N}$ and therefore is less than $P^2$.

\subsection{$M_{14}$} 
As $M_{14}$ is not idempotent we have no recourse to Theorem~\ref{thm:x2x}, but as with $M_{13}$, we may again find it as a quotient of the constant free large exponential algebra. 
Let $a$ denote those formal sums $s+\infty$ in ${\bf E}_-(x)$ with a nonzero linear summand, and $b$ those with no linear summand (we do not care if $\infty$ is present or not: so $x^2$ is a term with no linear summand, as is $x^2+\infty$).    
The properties in the table $M_{14}$ are then trivially verified.
The nonstandard model approach also can be applied.  Consider an infinite prime $P$ and the subalgebra of the nonstandard model ${\bf M}$ on $\overline{\mathbb{N}}$ generated by $P$, and let $a$ denote those elements that are not multiples of $P^2$, and~$b$ be everything else; this also returns $M_{14}$.

\subsection{$M_{15}$} 
This algebra is also not multiplicatively idempotent, but we may again use the constant-free large exponential algebra ${\bf E}_-(x)$. 
This time, let $a$ denote those formal sums $s+\infty$ in which $s$ has an odd coefficient of its linear component, and $b$ those with an even coefficient linear component.  This gives the algebra $M_{15}$.

%

\subsection{$M_{16}$} Mace4 finds this as a subreduct of a $4$-element multiplicatively idempotent HSI algebra.  
\[
\begin{tabular}{r|rrrr}
$+$ & $1$ & $2$ & $a$ & $b$\\
\hline
    $1$ & $2$ & $1$ & $b$ & $a$ \\
    $2$ & $1$ & $2$ & $a$ & $b$ \\
    $a$ & $b$ & $a$ & $a$ & $b$ \\
    $b$ & $a$ & $b$ & $b$ & $a$
\end{tabular} \hspace{.5cm}
\begin{tabular}{r|rrrr}
$\cdot$ & $1$ & $2$ & $a$ & $b$\\
\hline
    $1$ & $1$ & $2$ & $a$ & $b$ \\
    $2$ & $2$ & $2$ & $a$ & $a$ \\
    $a$ & $a$ & $a$ & $a$ & $a$ \\
    $b$ & $b$ & $a$ & $a$ & $b$
\end{tabular} \hspace{.5cm}
\begin{tabular}{r|rrrr}
$\uparrow$ & $1$ & $2$ & $a$ & $b$\\
\hline
    $1$ & $1$ & $1$ & $1$ & $1$ \\
    $2$ & $2$ & $2$ & $2$ & $2$ \\
    $a$ & $a$ & $a$ & $a$ & $a$ \\
    $b$ & $b$ & $b$ & $a$ & $a$
\end{tabular}
\]
As multiplicatively idempotent HSI algebras satisfy all valid identities of $\mathbb{N}$, it follows that $M_{16}$ satisfies all valid identities of $\mathbb{N}^-$.

\subsection{$M_{17}$} This is an HSI algebra reduct, isomorphic to algebra $5$ in Burris and Lee~\cite[Theorem~4.1]{burlee92} (the modulo $2$ quotient of ${\bf N}$).

\subsection{$M_{18}$} Mace4 finds this as a subreduct of the following $4$-element multiplicatively idempotent HSI~algebra:
\[ 
\begin{tabular}{r|rrrr}
$+$ & $1$ & $2$ & $a$ & $b$\\
\hline
    $1$ & $2$ & $1$ & $1$ & $2$ \\
    $2$ & $1$ & $2$ & $2$ & $1$ \\
    $a$ & $1$ & $2$ & $a$ & $b$ \\
    $b$ & $2$ & $1$ & $b$ & $a$
\end{tabular} \hspace{.5cm}
\begin{tabular}{r|rrrr}
$\cdot$ & $1$ & $2$ & $a$ & $b$\\
\hline
    $1$ & $1$ & $2$ & $a$ & $b$ \\
    $2$ & $2$ & $2$ & $a$ & $a$ \\
    $a$ & $2$ & $2$ & $a$ & $a$ \\
    $b$ & $1$ & $2$ & $a$ & $b$
\end{tabular} \hspace{.5cm}
\begin{tabular}{r|rrrr}
$\uparrow$ & $1$ & $2$ & $a$ & $b$\\
\hline
    $1$ & $1$ & $1$ & $1$ & $1$ \\
    $2$ & $2$ & $2$ & $1$ & $1$ \\
    $a$ & $a$ & $a$ & $b$ & $b$ \\
    $b$ & $b$ & $b$ & $b$ & $b$
\end{tabular}
\]
This last case completes the proof of Proposition~\ref{pro:2RSI}.
\end{proof}
At this point, the case of $4$-element HSI algebras appears to offer a better chance of progress than the case of $3$-element RSI algebras.
A natural question is whether every multiplicatively idempotent RSI algebra embeds into a multiplicatively idempotent HSI~algebra.  
This would extend Theorem~\ref{thm:x2x} to \RSI\ as well, but there are still more than 3 times as many 3-element  non-idempotent RSI algebras (187) as there are 4-element non-idempotent HSI algebras (60).
There are 216 distinct 3-element idempotent RSI algebras that are not already reducts of HSI algebras.  
As an initial computational exploration of the idempotent embedding question, we used Claude to manage scripting for a Mace4 search, finding that all 216 of these models as subreducts of an idempotent HSI algebra; it follows that none of these are Gurevi\u{c} algebras in the sense of RSI. 
Two examples required 6 further elements (with others requiring 0--5 further elements).
One of these examples is given here: the 3-element RSI algebra is the subalgebra on the elements $\{0,a,b\}$ (using $0$ because it is a multiplicative zero element, but noting that it is not an additive identity element).
\[\setlength{\arraycolsep}{3pt}
\renewcommand{\arraystretch}{0.9}
\begin{array}{c|ccccccccc}
+ & 0 & a & b & 1 & c & d & e & 2 & f\\
\hline
0 & 0 & 0 & b & c & c & c & e & e & e\\
a & 0 & a & b & d & c & d & e & f & f\\
b & b & b & 0 & e & e & e & c & c & c\\
1 & c & d & e & 2 & e & f & c & 1 & d\\
c & c & c & e & e & e & e & c & c & c\\
d & c & d & e & f & e & f & c & d & d\\
e & e & e & c & c & c & c & e & e & e\\
2 & e & f & c & 1 & c & d & e & 2 & f\\
f & e & f & c & d & c & d & e & f & f
\end{array}\quad
\begin{array}{c|ccccccccc}
\cdot & 0 & a & b & 1 & c & d & e & 2 & f\\
\hline
0 & 0 & 0 & 0 & 0 & 0 & 0 & 0 & 0 & 0\\
a & 0 & a & 0 & a & 0 & a & 0 & a & a\\
b & 0 & 0 & b & b & b & b & 0 & 0 & 0\\
1 & 0 & a & b & 1 & c & d & e & 2 & f\\
c & 0 & 0 & b & c & c & c & e & e & e\\
d & 0 & a & b & d & c & d & e & f & f\\
e & 0 & 0 & 0 & e & e & e & e & e & e\\
2 & 0 & a & 0 & 2 & e & f & e & 2 & f\\
f & 0 & a & 0 & f & e & f & e & f & f
\end{array}\quad
\begin{array}{c|ccccccccc}
{}^\wedge & 0 & a & b & 1 & c & d & e & 2 & f\\
\hline
0 & a & a & a & 0 & 0 & 0 & 0 & 0 & 0\\
a & a & a & a & a & a & a & a & a & a\\
b & a & a & a & b & 0 & 0 & 0 & b & 0\\
1 & 1 & 1 & 1 & 1 & 1 & 1 & 1 & 1 & 1\\
c & a & a & a & c & 0 & 0 & 0 & c & 0\\
d & a & a & a & d & a & a & a & d & a\\
e & a & a & a & e & 0 & 0 & 0 & e & 0\\
2 & a & a & a & 2 & a & a & a & 2 & a\\
f & a & a & a & f & a & a & a & f & a
\end{array}
\]

\section*{AI assistance}
Research in this paper was performed prior to the emergence of AI assisted reasoning, aside from the use of Prover9 and Mace4 as stated at points in the paper.  
ChatGPT Sol was used for proof reading, but not for writing.

\bibliographystyle{amsplain}


\end{document}